\documentclass[10pt]{amsart} 
\usepackage{amsmath,amssymb,bm, color} 

 \usepackage{float}
 
\usepackage{amsmath}
\usepackage{graphicx}

\usepackage{mathrsfs}
\usepackage{bbm}
\usepackage{bm}
\usepackage{amsfonts,amssymb}
\usepackage{multirow}
\usepackage{lineno}
\usepackage{color}

\numberwithin{equation}{section}
 \newtheorem{theorem}{Theorem}[section]
 \newtheorem{lemma}[theorem]{Lemma}

\def\3bar{{|\hspace{-.02in}|\hspace{-.02in}|}}
\def\E{{\mathcal{E}}}
\def\T{{\mathcal{T}}}

\def\cal#1{{\mathcal #1}}
\def\pT{{\partial T}}

\def\bw{{\mathbf{w}}}

\def\bI{{\mathbf{I}}}
\def\bf{{\mathbf{f}}}
\def\bu{{\mathbf{u}}}
\def\bg{{\mathbf{g}}}
\def\bv{{\mathbf{v}}}
\def\bn{{\mathbf{n}}}
\def\be{{\mathbf{e}}}

\def\beta{{\boldsymbol{\eta}}}
\def\bvarphi{{\boldsymbol{\varphi}}}

\def\bzeta{{\boldsymbol{\zeta}}}

\newtheorem{remark}{Remark}[section]
\newtheorem{algorithm}{Simplified Weak Galerkin Algorithm}[section]

\numberwithin{equation}{section}

\def\3bar{{|\hspace{-.02in}|\hspace{-.02in}|}}

 \def\cal#1{\mathcal{#1}}
\def\ad#1{\begin{aligned}#1\end{aligned}}  \def\b#1{\mathbf{#1}} 
 \def\an#1{\begin{align}#1\end{align}}  \def\p#1{\begin{pmatrix}#1\end{pmatrix}}

\begin{document}

\title []{Simplified Weak Galerkin Methods for Linear Elasticity on Nonconvex Domains}

  \author {Chunmei Wang}
  \address{Department of Mathematics, University of Florida, Gainesville, FL 32611, USA. }
  \email{chunmei.wang@ufl.edu}
  \thanks{The research of Chunmei Wang was partially supported by National Science Foundation Grant DMS-2136380.} 
 
\author {Shangyou Zhang}
\address{Department of Mathematical Sciences,  University of Delaware, Newark, DE 19716, USA}   \email{szhang@udel.edu}

\begin{abstract} 
This paper presents a weak Galerkin (WG) finite element method for linear elasticity on general polygonal and polyhedral meshes, free from convexity constraints,  by leveraging bubble functions as central analytical tools. The proposed method eliminates the need for stabilizers commonly used in traditional WG methods, resulting in a simplified formulation. The method is symmetric, positive definite, and straightforward to implement. Optimal-order error estimates are established for the WG approximations in the discrete $H^1$-norm, assuming sufficient smoothness of the exact solution, and in the standard $L^2$-norm under regularity assumptions for the dual problem. Numerical experiments confirm the efficiency and accuracy of the proposed stabilizer-free WG method.

\end{abstract}


 
\maketitle

\section{Introduction}
 This paper introduces a novel weak Galerkin finite element method for linear elasticity that eliminates the stabilizers traditionally required in WG methods. The innovation of the proposed approach, compared to existing stabilizer-free WG methods (e.g., \cite{ye}), lies in its effectiveness on general polygonal and polyhedral meshes, including those with nonconvex geometries, as it operates without requiring convexity assumptions.

Let $\Omega\subset\mathbb R^d$ $(d=2, 3)$ be an open, bounded, and connected domain with a Lipschitz continuous boundary $\partial \Omega$, representing  an elastic body subjected  to an exterior force $\bf$ and a prescribed  displacement boundary condition. The kinematic model of linear elasticity aims to determine  a displacement vector field $\bu$ satisfying 
\begin{equation}\label{model}
    \begin{split}
       -\nabla\cdot \sigma(\bu)&= \bf,\qquad \text{in}\ \Omega,  \\
  \bu&= \bg, \qquad \text{on}\ \partial \Omega,    
    \end{split}
\end{equation} 
 where $\sigma(\bu)$ is  the symmetric Cauchy stress tensor. For linear, homogeneous, and isotropic materials, the stress tensor is expressed as
$$
\sigma(\bu)=2\mu\epsilon(\bu)+\lambda (\nabla\cdot\bu) \bI,
$$
 with $\epsilon(\bu)=\frac{1}{2}(\nabla \bu+\nabla \bu^T)$ 
denoting the linear strain tensor, and $\mu$ and $\lambda$  representing  the Lam$\acute{e}$ constants. In the case of linear plane strain, the  Lam$\acute{e}$
constants are  defined as
$$
\lambda=\frac{E\nu}{(1+\nu)(1-2\nu)}, \qquad
\mu=\frac{E}{2(1+\nu)},
$$
where $E$ is the elasticity modulus, and $\nu$ is Poisson’s ratio.
 
The variational formulation of the model problem \eqref{model} can be formulated  as follows: Find an unknown function $\bu\in [H^1(\Omega)]^d$ such that $\bu=\bg$  on $\partial \Omega$ and
\begin{equation}\label{weak}
2(\mu \epsilon(\bu), \epsilon(\bv))+(\lambda \nabla\cdot\bu, \nabla\cdot\bv)=(\bf, \bv), \qquad \forall \bv\in [H_0^1(\Omega)]^d,
 \end{equation} 
  where $H_0^1(\Omega)=\{v\in H^1(\Omega): v=0 \ \text{on}\ \partial \Omega\}$.

  The Finite Element Method (FEM) and its variants are extensively used for solving partial differential equations(PDEs) numerically. In the context of elasticity problems, mixed FEMs are particularly popular; however, enforcing strong symmetry on the stress tensor presents a significant challenge. To address this, several strategies have been developed, including relaxing the symmetry constraint on the stress tensor \cite{57}, constructing weakly symmetric mixed finite elements \cite{13}, and designing nonconforming mixed FEMs \cite{4, 9, 12, 29, 62, 63, 64}. A breakthrough was achieved with a family of conforming mixed elements featuring reduced degrees of freedom, applicable in any dimension. This was accomplished by identifying a critical structure within the discrete stress spaces of symmetric matrix-valued polynomials on simplicial grids and establishing two fundamental algebraic results \cite{41}.
The Discontinuous Galerkin (DG) method has also been widely adopted for elasticity problems \cite{19, 59}. A key advantage of DG methods is their ability to discretize problems on an element-by-element basis, seamlessly connecting elements through numerical traces \cite{5, 31}. For linear elasticity, innovative approaches include a three-field decomposition method \cite{15} and a novel hybridized mixed method \cite{28}. Other noteworthy methods, such as the tangential-displacement  normal-stress method, have demonstrated robustness against both shear and volume locking \cite{48, 49}.

The weak Galerkin finite element method has revolutionized the numerical landscape for solving partial differential equations (PDEs). By harnessing the power of distributions and piecewise polynomials, WG transcends traditional finite element approaches. Unlike its predecessors, WG relaxes the stringent regularity requirements for function approximations, instead leveraging carefully crafted stabilizers to ensure method stability. Recent studies have exhaustively explored the versatility of WG in tackling diverse model PDEs, thereby casting the method as a robust and reliable tool in computational science \cite{wg1, wg2, wg3, wg4, wg5, wg6, wg7, wg8, wg9, wg10, wg11, wg12, wg13, wg14, wg15, wg16, wg17, wg18, wg19, wg20, wg21, itera, wy3655}. Its capacity to adapt to a wide range of PDEs is underscored by its use of weak derivatives and weak continuities in designing numerical schemes based on the weak forms of underlying PDEs.
A key advancement within the WG paradigm is the Primal-Dual Weak Galerkin (PDWG) method \cite{pdwg1, pdwg2, pdwg3, pdwg4, pdwg5, pdwg6, pdwg7, pdwg8, pdwg9, pdwg10, pdwg11, pdwg12, pdwg13, pdwg14, pdwg15}.  PDWG views numerical solutions as constrained minimizations of functionals, with constraints that mimic the weak formulation of PDEs using weak derivatives. This formulation results in an Euler-Lagrange equation that integrates both the primal variable and the dual variable (Lagrange multiplier), yielding a symmetric scheme. 
 
This paper presents a simplified formulation of the  WG  finite element method that eliminates the need for stabilizers. Unlike the existing stabilizer-free WG methods \cite{ye}, our approach is versatile, as it works on   polytopal meshes without convexity constraints, 
 and supports flexible polynomial degrees. The critical innovation facilitating these advancements is the incorporation of bubble functions. While this method requires the use of higher-degree polynomials for calculating the discrete weak derivatives, which may pose challenges for certain practical applications, our focus is on the theoretical advancements, particularly in developing WG methods with inherent stabilizers, specifically designed for non-convex elements in finite element partitions.

Our method retains the size and global sparsity of the stiffness matrix, greatly simplifying programming complexity in comparison to traditional stabilizer-dependent WG methods. Theoretical analysis confirms that our WG approximations yield optimal error estimates in both the discrete 
$H^1$
 and 
$L^2$
 norms. By providing a stabilizer-free WG method that sustains high performance while reducing computational complexity, this paper makes a significant contribution to the development of finite element methods on non-convex polytopal meshes.

 The structure of this paper is as follows. Section 2 provides a concise review of weak differential operators and their discrete counterparts. In Section 3, we present the weak Galerkin scheme that eliminates the need for stabilizers. Section 4 proves the existence and uniqueness of the solution. In Section 5, we derive the error equation for the proposed scheme, followed by Section 6, which focuses on the error estimate for the numerical approximation in the energy norm. Section 7 extends this analysis to establish the error estimate in the 
$L^2$ norm. Finally, Section 8 presents numerical tests to validate the theoretical findings from the previous sections.

Throughout this paper, we adopt standard notations. Let $D$ be any open bounded domain   in $\mathbb{R}^d$ with a Lipschitz continuous boundary.  The inner product, semi-norm and norm in the Sobolev space $H^s(D)$ for any integer $s\geq0$ are denoted by $(\cdot,\cdot)_{s,D}$, $|\cdot|_{s,D}$ and $\|\cdot\|_{s,D}$, respectively. For simplicity, when the domain $D$ is chosen as $D=\Omega$, the subscript $D$ is  omitted from the notations of the inner product and norm. For the case where $s=0$, the notations $(\cdot,\cdot)_{0,D}$, $|\cdot|_{0,D}$ and $\|\cdot\|_{0,D}$ are further simplified as $(\cdot,\cdot)_D$, $|\cdot|_D$ and $\|\cdot\|_D$, respectively.

\section{Discrete Weak Strain Tensor and Discrete Weak Divergence}\label{Section:Hessian}
In this section, we briefly review the definition of weak strain tensor and weak divergence, along with their discrete counterparts introduced in \cite{wg10, wg18}.

Let $T$ be a polytopal element with boundary $\partial T$. A weak function on $T$  is defined as   $\bv=\{\bv_0, \bv_b\}$, where $\bv_0\in [L^2(T)]^d$, $\bv_b\in [L^{2}(\partial T)]^d$. The first component, $\bv_0$, represents the value of $\bv$ in the interior of $T$, while the second component, $\bv_b$, corresponds to  the value of $\bv$   on the boundary of $T$. In general, $\bv_b$ is assumed to be independent of the trace of $\bv_0$. A special case arises when $\bv_b= \bv_0|_{\partial T}$, where the function $
\bv=\{\bv_0, \bv_b\}$ is fully determined by $\bv_0$ and can be simply denoted as $\bv=\bv_0$.
 
 Denote by $W(T)$ the space of all weak functions on $T$; i.e.,
 \begin{equation*}\label{2.1}
 W(T)=\{v=\{\bv_0,\bv_b\}: \bv_0\in [L^2(T)]^d, \bv_b\in [L^{2}(\partial
 T)]^d\}.
\end{equation*}
 
 The weak gradient, denoted by $\nabla_{w}$, is a linear
 operator from $W(T)$ to the dual space of $[H^{1}(T)]^{d\times d}$. For any
 $\bv\in W(T)$, the weak gradient  $\nabla_w \bv$ is defined as a bounded linear functional on $[H^{1}(T)]^{d\times d}$
such that
 \begin{equation*}\label{2.3}
  (\nabla _{w}\bv, \bvarphi)_T=-(\bv_0,\nabla\cdot \bvarphi)_T+
  \langle \bv_b, \bvarphi\cdot \bn \rangle_{\partial T},\quad \forall \bvarphi\in [H^{1}(T)]^{d\times d},
  \end{equation*}
 where $ \bn$ is an unit outward normal direction to $\partial T$.
 
 For any non-negative integer $r$, let $P_r(T)$  represent the space of
 polynomials on $T$ with total degree at most 
 $r$. A discrete weak
 gradient on $T$, denoted by $\nabla_{w, r_1, T}$, is a linear operator
 from $W(T)$ to $[P_{r_1}(T)]^{d\times d}$. For any $\bv\in W(T)$,
 $\nabla_{w, r_1, T}\bv$ is the unique polynomial matrix in $[P_{r_1}(T)]^{d\times d}$ satisfying
 \begin{equation*}\label{2.4}
  (\nabla_{w,r_1,T} \bv, \bvarphi)_T=-(\bv_0,\nabla\cdot \bvarphi)_T+
  \langle \bv_b,  \bvarphi\cdot \bn \rangle_{\partial T},\quad \forall \bvarphi \in [P_{r_1}(T)]^{d\times d}.
  \end{equation*}

We define the discrete weak strain tensor  as follows:
$$
\epsilon_{w,r_1,T}(\bu)=\frac{1}{2}(\nabla_{w,r_1,T}\bu+\nabla_{w,r_1,T}\bu^T).
$$
 
 For any $\bv\in W(T)$, the discrete weak strain tensor, denoted by
 $\epsilon_{w, r_1, T}(\bv)$,  is the unique polynomial matrix in $[P_{r_1}(T)]^{d\times d}$ satisfying
 \begin{equation}\label{2.5}
  (\epsilon_{w,r_1,T} (\bv), \bvarphi)_T=-(\bv_0,\nabla  \cdot \frac{1}{2}(\bvarphi+\bvarphi^T) )_T+
  \langle \bv_b,  \frac{1}{2}(\bvarphi+\bvarphi^T)\cdot \bn \rangle_{\partial T}, 
  \end{equation}
   for all $\bvarphi \in [P_{r_1}(T)]^{d\times d}$.
  
 For a smooth $\bv_0\in
 [H^1(T)]^d$,  applying the usual integration by parts to the first
 term on the right-hand side of (\ref{2.5})  gives
 \begin{equation}\label{2.5new}
 (\epsilon_{w,r_1,T} (\bv), \bvarphi)_T=(\epsilon( \bv_0), \bvarphi)_T+
  \langle \bv_b-\bv_0, \frac{1}{2}( \bvarphi+\bvarphi^T)\cdot \bn \rangle_{\partial T}.  
  \end{equation} 
 for all $\bvarphi \in [P_{r_1}(T)]^{d\times d}$.

 The weak divergence of $\bv\in W(T)$, denoted by $\nabla_w\cdot \bv$, is a bounded linear functional in the Sobolev space $H^1(T)$, and its action on any $\phi\in H^1(T)$ is given by
\begin{equation*}\label{div}
    (\nabla_w\cdot \bv, \phi)_T=-(\bv_0, \nabla\phi)_T+\langle \bv_b\cdot\bn, \phi\rangle_{\partial T}.
\end{equation*}

The discrete weak divergence of $\bv\in W(T)$, denoted by $\nabla_{w, r_2, T}\cdot \bv$, is the unique polynomial in $P_{r_2}(T)$ satisfying 
\begin{equation}\label{disdiv}
    (\nabla_{w, r_2, T}\cdot \bv, \phi)_T=-(\bv_0, \nabla\phi)_T+\langle \bv_b\cdot\bn, \phi\rangle_{\partial T},
\end{equation}
for any $\phi\in P_{r_2}(T)$. 

 For a smooth $\bv_0\in
 [H^1(T)]^d$,  applying the usual integration by parts to the first
 term on the right-hand side of (\ref{disdiv})  gives
\begin{equation}\label{disdivnew}
    (\nabla_{w, r_2, T}\cdot \bv, \phi)_T= (\nabla\cdot \bv_0,  \phi)_T+\langle (\bv_b-\bv_0)\cdot\bn, \phi\rangle_{\partial T},
\end{equation}
for any $\phi\in P_{r_2}(T)$.

\section{Weak Galerkin Algorithms without Stabilizers}\label{Section:WGFEM}
 Let ${\cal T}_h$ be a finite element partition of the domain
 $\Omega\subset \mathbb R^d$ into polytopal elements, where   ${\cal
 T}_h$ is assumed to be shape-regular as defined in   \cite{wy3655}.
 Denote by ${\mathcal E}_h$ the set of all edges/faces  in
 ${\cal T}_h$, and let ${\mathcal E}_h^0={\mathcal E}_h \setminus
 \partial\Omega$ be the set of interior edges/faces. The diameter of an element  $T\in {\cal T}_h$ is denoted
 by $h_T$, and the mesh size of the partition is given by $h=\max_{T\in {\cal
 T}_h}h_T$.


 For each element $T\in\T_h$, we define the local weak finite element space as:
 \begin{equation}\label{Vk}
 V(k, T)=\{\{\bv_0,\bv_b\}: \bv_0\in [P_k(T)]^d,\bv_b\in [P_{k}(e)]^d\}.   
 \end{equation}
By assembling $V(k, T)$ over all the elements $T\in {\cal T}_h$ and enforcing continuity on the interior interfaces $\E_h^0$,
 we define the global weak finite element space:
\begin{equation}\label{Vh}
 V_h=\big\{\{\bv_0,\bv_b\}:\ \{\bv_0,\bv_b\}|_T\in V(k,  T),
 \forall T\in {\cal T}_h \big\}.
 \end{equation}
Additionally, we denote by $V_h^0$ the subspace of $V_h$ with vanishing boundary value on $\partial\Omega$:
\begin{equation}\label{Vh0}
V_h^0=\{\{\bv_0,\bv_b\}\in V_h: \bv_b=0 \ \text{on}\ \partial\Omega\}.
\end{equation}

For simplicity,  the discrete weak strain tensor $\epsilon_{w, r_1, T}\bv$ and the discrete weak divergence $\nabla_{w, r_2, T} \cdot\bv$  are denoted by  
  $\epsilon_{w}\bv$ and $\nabla_{w} \cdot\bv$, respectively. These quantities are computed locally on each element $T$ using definitions 
\eqref{2.5} and \eqref{disdiv}: 
$$
(\epsilon_{w} \bv)|_T= \epsilon_{w, r_1, T}(\bv |_T), \qquad \forall T\in \T_h,
$$ 
$$
(\nabla_{w}\cdot \bv)|_T= \nabla_{w, r_2, T}\cdot(\bv |_T), \qquad \forall T\in \T_h.
$$

 The WG numerical scheme without stabilizers,  based on the weak formulation \eqref{weak} for the elasticity problem   \eqref{model},  is as follows:
  \begin{algorithm}\label{PDWG1}
 Find $\bu_h=\{\bu_0, \bu_b\} \in V_h$  such that    $\bu_b=Q_b\bg$ on $\partial\Omega$ and  
 \begin{equation}\label{WG}
\sum_{T\in {\cal T}_h} 2(\mu \epsilon_w(\bu_h), \epsilon_w(\bv))_T+(\lambda \nabla_w \cdot \bu_h, \nabla_w\cdot\bv)_T=\sum_{T\in {\cal T}_h}(\bf, \bv_0)_T, 
 \end{equation}
for all  $\bv=\{\bv_0, \bv_b\}\in V_h^0$. Here $Q_b$ denotes the $L^2$ projection operator onto the space $P_k(e)$.
 \end{algorithm}

This scheme directly solves the elasticity problem without introducing any stabilization terms.

\section{Solution Existence and Uniqueness} 
 
To begin, we recall the essential trace inequalities. Given that ${\cal T}_h$ is a shape-regular finite element partition of the domain $\Omega$, the following trace inequality holds for any element  $T\in {\cal T}_h$ and function $\phi\in H^1(T)$ \cite{wy3655}: 

\begin{equation}\label{tracein}
 \|\phi\|^2_{\partial T} \leq C(h_T^{-1}\|\phi\|_T^2+h_T \|\nabla \phi\|_T^2).
\end{equation}
For polynomials $\phi$, a simplified trace inequality is used  \cite{wy3655}: 
\begin{equation}\label{trace}
\|\phi\|^2_{\partial T} \leq Ch_T^{-1}\|\phi\|_T^2.
\end{equation}

Next, we define two norms that are crucial in the error analysis. 
For any $\bv=\{\bv_0, \bv_b\}\in V_h$, we define
the following discrete energy norm \begin{equation}\label{3norm}
\3bar \bv\3bar=\Big( \sum_{T\in {\cal T}_h} (2\mu\epsilon_w (\bv), \epsilon_w(\bv))_T+(\lambda \nabla_w \cdot \bv, \nabla_w \cdot \bv)_T\Big)^{\frac{1}{2}},
\end{equation}
and the following discrete $H^1$ semi-norm 
\begin{equation}\label{disnorm}
\|\bv\|_{1, h}=\Big(\sum_{T\in {\cal T}_h} (2\mu\epsilon  (\bv_0), \epsilon (\bv_0))_T+(\lambda \nabla  \cdot \bv_0, \nabla  \cdot \bv_0)_T+h_T^{-1}\|\bv_0-\bv_b\|_{\partial T}^2\Big)^{\frac{1}{2}}.
\end{equation}
\begin{lemma}\cite{wang1}\label{norm1}
 For $\bv=\{\bv_0, \bv_b\}\in V_h$, there exists a constant $C$ such that
 $$
 \|\epsilon(\bv_0)\|_T\leq C\|\epsilon_w (\bv)\|_T.
 $$
\end{lemma}
\begin{proof}  Let  $T\in {\cal T}_h$ be a polytopal element with $N$ edges/faces denoted by $e_1, \cdots, e_N$. It is important to emphasis that the polytopal element $T$  can be non-convex. For each edge/face $e_i$, we construct   a linear equation  $l_i(x)$ such that  $l_i(x)=0$ on $e_i$,  defined as follows:
$$l_i(x)=\frac{1}{h_T}\overrightarrow{AX}\cdot \bn_i, $$  where  $A=(A_1, \cdots, A_{d-1})$ is a given point on the edge/face $e_i$,  $X=(x_1, \cdots, x_{d-1})$ is an arbitrary point on the edge/face $e_i$, $\bn_i$ is the normal direction to the edge/face $e_i$, and $h_T$ represents the size of the element $T$. 

The bubble function of  the element  $T$ can be  defined as 
 $$
 \Phi_B =l^2_1(x)l^2_2(x)\cdots l^2_N(x) \in P_{2N}(T).
 $$ 
 It is straightforward to verify that  $\Phi_B=0$ on the boundary $\partial T$.    The function 
  $\Phi_B$  can be scaled so that $\Phi_B(M)=1$ where   $M$  represents the barycenter of the element $T$. Additionally,  there exists a sub-domain $\hat{T}\subset T$ such that $\Phi_B\geq \rho_0$ for some constant $\rho_0>0$.
  
For $\bv=\{\bv_0, \bv_b\}\in V_h$, let $r_1=2N+k-1$ and set 
$\bvarphi=\Phi_B \epsilon (\bv_0)\in [P_{r_1}(T)]^{d\times d}$ in \eqref{2.5new}. We then obtain: 
\begin{equation}\label{t1}
\begin{split}
&(\epsilon_w (\bv), \Phi_B \epsilon (\bv_0))_T\\ =&(\epsilon(\bv_0), \Phi_B \epsilon (\bv_0))_T+\langle \bv_b-\bv_0,  \frac{1}{2}(\Phi_B \epsilon (\bv_0)+\Phi_B \epsilon (\bv_0)^T) \cdot \bn\rangle_{\partial T}\\ =&(\epsilon(\bv_0), \Phi_B \epsilon (\bv_0))_T,
\end{split}
\end{equation}
where we used $\Phi_B=0$ on $\partial T$.

From the domain inverse inequality \cite{wy3655},  there exists a constant $C$ such that 
\begin{equation}\label{t2}
(\epsilon (\bv_0), \Phi_B \epsilon (\bv_0))_T \geq C (\epsilon (\bv_0), \epsilon (\bv_0))_T.
\end{equation} 
Using the  Cauchy-Schwarz inequality along with  \eqref{t1}-\eqref{t2}, we have
 $$
 (\epsilon (\bv_0), \epsilon (\bv_0))_T\leq C (\epsilon_w (\bv), \Phi_B \epsilon (\bv_0))_T  \leq C  \|\epsilon_w (\bv)\|_T \|\Phi_B \epsilon (\bv_0)\|_T  \leq C
\|\epsilon_w (\bv)\|_T \|\epsilon (\bv_0)\|_T,
 $$
which implies
 $$
 \|\epsilon (\bv_0)\|_T\leq C\|\epsilon_w (\bv)\|_T.
 $$

This completes the proof of the lemma.
\end{proof}

\begin{lemma}\label{norm2}
 For $\bv=\{\bv_0, \bv_b\}\in V_h$, there exists a constant $C$ such that
 $$
 \|\nabla\cdot\bv_0\|_T\leq C\|\nabla_w\cdot \bv\|_T.
 $$
\end{lemma}
\begin{proof}
The proof follows the same approach as in Lemma  \ref{norm1}.
\end{proof} 

\begin{remark}
   If the polytopal element $T$  is convex, 
   the bubble function  in Lemma \ref{norm1}  can be  simplified to
 $$
 \Phi_B =l_1(x)l_2(x)\cdots l_N(x).
 $$ 
It can be verified that there exists a sub-domain $\hat{T}\subset T$,  such that
 $ \Phi_B\geq\rho_0$  for some constant $\rho_0>0$,  and $\Phi_B=0$ on the boundary $\partial T$.   Lemmas \ref{norm1}-\ref{norm2}   can be proved in the same manner using this simplified construction. In this case, we take $r_1=N+k-1$ and $r_2=N+k-1$.  
\end{remark}

Recall that $T$ is a $d$-dimensional polytopal element and  $e_i$ is a $(d-1)$-dimensional edge/face  of $T$. 
We construct an edge/face-based bubble function   $$\varphi_{e_i}= \Pi_{k=1, \cdots, N, k\neq i}l_k^2(x).$$ It can be verified that  (1) $\varphi_{e_i}=0$ on the edge/face $e_k$ for $k \neq i$, (2) there exists a subdomain $\widehat{e_i}\subset e_i$ such that $\varphi_{e_i}\geq \rho_1$ for some constant $\rho_1>0$.

\begin{lemma}\label{phi}
     For $\bv=\{\bv_0, \bv_b\}\in V_h$, let $\bvarphi=(\bv_b-\bv_0) \bn^T\varphi_{e_i}$, where $\bn$ is the unit outward normal direction to the edge/face  $e_i$. The following inequality holds:
\begin{equation}
  \|\bvarphi\|_T ^2 \leq Ch_T \int_{e_i}((\bv_b-\bv_0)\bn^T)^2ds.
\end{equation}
\end{lemma}
\begin{proof}
 We first extend $\bv_b$,  defined on the $(d-1)$-dimensional edge/face  $e_i$, to the entire d-dimensional polytopal element $T$ and claim that $\bv_b$ remains  a polynomial vector  on  $T$ after extension. 
Next, let $\bv_{trace}$ denote the trace of $\bv_0$ on   $e_i$ and extend $\bv_{trace}$   to the entire element $T$. 
Like $\bv_b$, $\bv_{trace}$ remains a polynomial after the  extension.  For details on these extensions, see \cite{wang1, wang2}.

Now, define $\bvarphi=(\bv_b-\bv_0) \bn^T\varphi_{e_i}$. We have
\begin{equation*}
    \begin{split}
\|\bvarphi\|^2_T  =
\int_T \bvarphi^2dT =  &\int_T ((\bv_b-\bv_{trace})  \bn^T\varphi_{e_i})^2dT\\
\leq &Ch_T \int_{e_i} ((\bv_b-\bv_{trace})  \bn^T\varphi_{e_i})^2ds\\
\leq &Ch_T \int_{e_i} ((\bv_b-\bv_0)\bn^T)^2ds,
    \end{split}
\end{equation*} 
where we used  (1) $\varphi_{e_i}=0$ on the edge/face $e_k$ for $k \neq i$, (2) there exists a subdomain $\widehat{e_i}\subset e_i$ such that $\varphi_{e_i}\geq \rho_1$ for some constant $\rho_1>0$, and applied the properties of the projection.

 This completes the proof of the lemma.

\end{proof}

\begin{lemma}\label{phi2}
     For $\bv=\{\bv_0, \bv_b\}\in V_h$, let $\phi=(\bv_b-\bv_0) \cdot\bn \varphi_{e_i}$, where $\bn$ is the unit outward normal direction to the edge/face  $e_i$. The following inequality holds:
\begin{equation}
  \|\phi\|_T ^2 \leq Ch_T \int_{e_i}(\bv_b-\bv_0)^2ds.
\end{equation}
\end{lemma}
\begin{proof}
This proof follows similarly to Lemma  \ref{phi}.
\end{proof} 
\begin{lemma}\label{normeqva}   There exists  positive constants $C_1$ and $C_2$ such that for any $\bv=\{\bv_0, \bv_b\} \in V_h$, we have
 \begin{equation}\label{normeq}
 C_1\|\bv\|_{1, h}\leq \3bar \bv\3bar  \leq C_2\|\bv\|_{1, h}.
\end{equation}
\end{lemma}

\begin{proof}   
 Note that the polytopal element $T$ can be non-convex.
Recall that an edge/face-based bubble function  is defined as  $$\varphi_{e_i}= \Pi_{k=1, \cdots, N, k\neq i}l_k^2(x).$$

We extend $\bv_b$ from $e_i$ to $T$,  and similarly extend $\bv_{trace}$ (trace of $\bv_0$ on $e_i$) to $T$. We denote these extensions as  $\bv_b$ and $\bv_0$ for simplicity. Details of the extensions can be found in Lemma \ref{phi} and \cite{wang1, wang2}.

By choosing $\bvarphi=(\bv_b-\bv_0) \bn^T\varphi_{e_i}$ in \eqref{2.5new}, we have: 
\begin{equation}\label{t3} 
\begin{split}
   (\epsilon_{w} (\bv), \bvarphi)_T=&(\epsilon (\bv_0), \bvarphi)_T+
   \langle \bv_b-\bv_0,  \frac{1}{2}(\bvarphi+\bvarphi^T)\cdot \bn \rangle_{\partial T}\\  =&(\epsilon (\bv_0), \bvarphi)_T+ \int_{e_i}|\bv_b-\bv_0|^2 \varphi_{e_i}ds, 
\end{split}
  \end{equation} 
   where we used  $\varphi_{e_i}=0$ on  $e_k$ for $k \neq i$ and the fact that   $\varphi_{e_i}\geq \rho_1$ for some constant $\rho_1>0$ in a subdomain $\widehat{e_i}\subset e_i$.

   Using Cauchy-Schwarz inequality, \eqref{t3}, the domain inverse inequality \cite{wy3655}, and  Lemma \ref{phi} gives
\begin{equation*}
\begin{split}
 \int_{e_i}|\bv_b-\bv_0|^2  ds\leq &C  \int_{e_i}|\bv_b-\bv_0|^2  \varphi_{e_i}ds \\
 \leq & C(\|\epsilon_w (\bv)\|_T+\|\epsilon (\bv_0)\|_T)\| \bvarphi\|_T\\
 \leq & {Ch_T^{\frac{1}{2}} (\|\epsilon_w (\bv)\|_T+\|\epsilon (\bv_0)\|_T) (\int_{e_i}((\bv_0-\bv_b) )^2ds)^{\frac{1}{2}}},
 \end{split}
\end{equation*}
which, from Lemma \ref{norm1}, gives 
\begin{equation}\label{a1}
 h_T^{-1}\int_{e_i}|\bv_b-\bv_0|^2  ds \leq C  (\|\epsilon_w (\bv)\|^2_T+\|\epsilon (\bv_0)\|^2_T)\leq C\|\epsilon_w (\bv)\|^2_T.    
\end{equation}

Choosing $\phi=(\bv_b-\bv_0) \cdot\bn\varphi_{e_i}$ in \eqref{disdivnew}, gives
\begin{equation}\label{t4}
\begin{split}
   (\nabla_{w} \cdot \bv, \phi)_T&=(\nabla\cdot \bv_0, \phi)_T+
   \langle (\bv_b-\bv_0)\cdot \bn,  \phi \rangle_{\partial T} \\& =(\nabla\cdot \bv_0, \phi)_T+ \int_{e_i}|\bv_b-\bv_0|^2 \varphi_{e_i}ds,  
\end{split}
  \end{equation} 
   where we used (1) $\varphi_{e_i}=0$ on the edge/face $e_k$ for $k \neq i$, 
(2) there exists a subdomain $\widehat{e_i}\subset e_i$ such that $\varphi_{e_i}\geq \rho_1$ for some constant $\rho_1>0$.

   Using Cauchy-Schwarz inequality, \eqref{t4}, the domain inverse inequality \cite{wy3655}, and  Lemma \ref{phi2} gives
\begin{equation*}
\begin{split}
 \int_{e_i}|\bv_b-\bv_0|^2  ds\leq &C  \int_{e_i}|\bv_b-\bv_0|^2  \varphi_{e_i}ds \\
 \leq & C(\|\nabla_w \cdot \bv\|_T+\|\nabla\cdot\bv_0\|_T)\| \phi\|_T\\
 \leq & {Ch_T^{\frac{1}{2}} (\|\nabla_w \cdot  \bv \|_T+\|\nabla\cdot \bv_0\|_T) (\int_{e_i}(\bv_0-\bv_b)^2ds)^{\frac{1}{2}}},
 \end{split}
\end{equation*}
which, from Lemma \ref{norm2}, gives 
\begin{equation}\label{t5}
     h_T^{-1}\int_{e_i}|\bv_b-\bv_0|^2  ds \leq C  (\|\nabla_w \cdot  \bv\|^2_T+\|\nabla\cdot \bv_0\|^2_T)\leq C\|\nabla_w \cdot  \bv\|^2_T.
\end{equation}

This, together with \eqref{a1}, \eqref{t5}, Lemmas \ref{norm1}-\ref{norm2},  \eqref{3norm} and \eqref{disnorm}, gives
$$
 C_1\|\bv\|_{1, h}\leq \3bar \bv\3bar.
$$

Next, from \eqref{2.5new}, Cauchy-Schwarz inequality and  the trace inequality \eqref{trace}, we conclude:
$$
 \Big|(\epsilon_{w} (\bv), \bvarphi)_T\Big| \leq \|\epsilon (\bv_0)\|_T \|  \bvarphi\|_T+
Ch_T^{-\frac{1}{2}}\|\bv_b-\bv_0\|_{\partial T} \| (\bvarphi+\bvarphi^T)\cdot\bn\|_{T},
$$
which yields
\begin{equation}\label{a2}
\| \epsilon_{w} (\bv)\|_T^2\leq C( \|\epsilon (\bv_0)\|^2_T  +
 h_T^{-1}\|\bv_b-\bv_0\|^2_{\partial T}).
\end{equation}

From \eqref{disdivnew}, Cauchy-Schwarz inequality and  the trace inequality \eqref{trace}, we have
$$
 \Big|(\nabla_{w} \cdot \bv, \phi)_T\Big| \leq \|\nabla\cdot \bv_0 \|_T \|  \phi\|_T+
Ch_T^{-\frac{1}{2}}\|(\bv_b-\bv_0)\cdot\bn\|_{\partial T} \| \phi \|_{T},
$$
which yields
\begin{equation}\label{a3}
\| \nabla_{w}\cdot \bv\|_T^2\leq C( \|\nabla
\cdot\bv_0\|^2_T  +
 h_T^{-1}\|\bv_b-\bv_0\|^2_{\partial T}),    
\end{equation}

Using \eqref{a2}-\eqref{a3}, \eqref{3norm} and \eqref{disnorm}, gives $$ \3bar \bv\3bar  \leq C_2\|\bv\|_{1, h}.$$

 This completes the proof of the lemma.
 \end{proof}

  \begin{remark}
   If the polytopal element $T$  is convex, 
  the edge/face-based bubble function in Lemmas \ref{phi}-\ref{normeqva}  simplifies to
$$\varphi_{e_i}= \Pi_{k=1, \cdots, N, k\neq i}l_k(x).$$
It can be verified that (1)  $\varphi_{e_i}=0$ on the edge/face $e_k$ for $k \neq i$, (2) there exists a subdomain $\widehat{e_i}\subset e_i$ such that $\varphi_{e_i}\geq \rho_1$ for some constant $\rho_1>0$. Lemmas \ref{phi}-\ref{normeqva}  can be derived similarly using this simplified construction.
\end{remark}

\begin{remark}
For any  $d$-dimensional polytopal element $T$,  
 there exists a hyperplane $H\subset R^d$  such that a finite number $l$ of distinct $(d-1)$-dimensional edges/faces containing $e_{i}$ are contained within $H$.  
 In such cases, Lemmas \ref{phi}-\ref{normeqva} can be proved with additional techniques. For more details, see \cite{wang1, wang2}, which can be generalized to Lemmas  \ref{phi}-\ref{normeqva}.
 \end{remark}

\begin{lemma}\cite{wg10, wg18} (Second Korn's Inequality)
Let $\Omega$  be a connected, open, bounded domain with a Lipschitz continuous boundary. Assume $\Gamma_1\subset \partial\Omega$ is a nontrivial portion of the boundary $\partial\Omega$, with dimension $d-1$. For any fixed real number $1\leq p<\infty$, there exists a constant $C$ such that
\begin{equation}\label{korn}
 \|\bv\|_1\leq C(\|\epsilon(\bv)\|_0+\|\bv\|_{L^p(\Gamma_1)}), 
\end{equation}
for any $\bv\in [H^1(\Omega)]^d$.
\end{lemma}
 
\begin{theorem}
The  Weak Galerkin  Algorithm without Stabilizers \ref{PDWG1} has  a unique solution. 
\end{theorem}
\begin{proof} 
 Since the number of equations equals the number of unknowns in \eqref{WG}, it suffices to prove the uniqueness of the solution.
Assume $\bu_h^{(1)}\in V_h$ and $\bu_h^{(2)}\in V_h$ are two distinct  solutions of the WG Algorithm  \ref{PDWG1}. Define $\beta_h= \bu_h^{(1)}-
\bu_h^{(2)}\in V_h^0$. Then, $\beta_h$ satisfies the equation
$$
\sum_{T\in {\cal T}_h} 2(\mu \epsilon_w(\beta_h), \epsilon_w(\bv))_T+(\lambda \nabla_w \cdot \beta_h, \nabla_w\cdot\bv)_T=0.
$$
By setting $\bv=\beta_h$,  we obtain $\3bar \beta_h\3bar=0$. Using the fact that $\3bar \beta_h\3bar=0$ and \eqref{normeq}, we have
$$\|\beta_h\|_{1,h}=0.$$ 
This implies that $\epsilon(\beta_0)=0$ and $\nabla\cdot\beta_0=0$ on each element $T$, and that $\beta_0=\beta_b$ on each $\partial T$.

Since $\epsilon(\beta_0)=0$ on each element  $T$, it follows that $\beta_0\in RM(T)\subset [P_1(T)]^d$. Consequently, $\beta_0=\beta_b$ on each $\partial T$ and $\beta_0$ is continuous across the domain $\Omega$.  Using the fact that  $\beta_b=0$ on $\partial\Omega$, we conclude that $\beta_0=0$ on $\partial\Omega$. Applying the second Korn's inequality \eqref{korn}, we conclude that $\beta_0\equiv 0$ in $\Omega$. Since   $\beta_0=\beta_b$ on each $\partial T$, we have $\beta_b\equiv 0$ in $\Omega$. Therefore,  $\bu_h^{(1)}=\bu_h^{(2)}$. 

Finally, note that this result holds for any $\lambda\geq 0$, completing the proof. 
\end{proof}

\section{Error Equations}
On each element $T\in\T_h$, let $Q_0$ denote the $L^2$ projection onto $P_k(T)$. On each  edge/face  $e\subset\partial T$, recall that $Q_b$ is the $L^2$ projection operator onto $P_k(e)$. For any $\bw\in [H^1(\Omega)]^d$, we define the $L^2$ projection into the weak finite element space $V_h$, denote by $Q_h \bw$, such that
$$
(Q_h\bw)|_T:=\{Q_0(\bw|_T),Q_b(\bw|_{\pT})\},\qquad \forall T\in\T_h.
$$
We also denote by $Q_{r_1}$ and $Q_{r_2}$ the $L^2$ projection operators onto the finite element spaces of piecewise polynomials of degrees $r_1$ and $r_2$ respectively.

\begin{lemma} \label{pror1}The following property holds:
\begin{equation}\label{pro}
\epsilon_{w}(\bw) =Q_{r_1}\epsilon (\bw), \qquad \forall \bw\in [H^1(T)]^d.
\end{equation}
\end{lemma}

\begin{proof} For any $\bw\in [H^1(T)]^d$,  using \eqref{2.5new}, we obtain 
 \begin{equation*} 
 \begin{split}
  (\epsilon_{w} (\bw),\bvarphi)_T=&(\epsilon (\bw), \bvarphi)_T+
  \langle \bw|_{\partial T}-\bw|_T, \frac{1}{2}(\bvarphi+\bvarphi^T) \cdot \bn \rangle_{\partial T}\\=&(\epsilon (\bw), \bvarphi)_T\\
  =&(Q_{r_1}\epsilon (\bw), \bvarphi)_T, 
  \end{split}
  \end{equation*}   
  for any $\bvarphi\in [P_{r_1}(T)]^{d\times d}$.  This 
     completes the proof of this lemma.
  \end{proof}

\begin{lemma} \label{pror2}   The following property holds:
\begin{equation}\label{pro2}
\nabla_w\cdot \bw  =Q_{r_2}(\nabla \cdot  \bw), \qquad \forall \bw\in [H^1(T)]^d.
\end{equation}
\end{lemma}
\begin{proof}
 The proof of this lemma follows similarly to that of Lemma \ref{pror1}.   
\end{proof}

Let  $\bu$ and $\bu_h \in V_{h}$ denote the exact solution of the elasticity problem \eqref{model} and its numerical approximation arising from the WG Algorithm  \ref{PDWG1}, respectively. We define the error function $\be_h$ as follows
\begin{equation}\label{error} 
\be_h=\bu-\bu_h  \in V_{h}^0.
\end{equation}

\begin{lemma}\label{errorequa}
The error function $\be_h$ given in (\ref{error}) satisfies the following error equation:
\begin{equation}\label{erroreqn}
\sum_{T\in {\cal T}_h}(2\mu\epsilon_w (\be_h), \epsilon_w (\bv))_T+(\lambda \nabla_w \cdot \be_h, \nabla_w \cdot \bv)_T =\ell (\bu, \bv), \qquad \forall \bv\in V_h^0,
\end{equation}
where 
$$
\ell (\bu, \bv)=\sum_{T\in {\cal T}_h}  \langle   \bv_b-\bv_0, 2\mu( Q_{r_1}-I)  \epsilon (\bu) \cdot\bn\rangle_{\partial T} +\langle (\bv_b-\bv_0)\cdot\bn, \lambda   (Q_{r_2} -I)(\nabla  \cdot \bu)\rangle_{\partial T}.
$$
\end{lemma}
\begin{proof}   Using \eqref{pro}, \eqref{pro2}, and letting $\bvarphi= Q_{r_1} \epsilon(\bu)$ in \eqref{2.5new} and $\phi=Q_{r_2} (\nabla\cdot\bu)$ in \eqref{disdivnew}, gives
\begin{equation}\label{54}
\begin{split}
&\sum_{T\in {\cal T}_h}(2\mu\epsilon_w (\bu), \epsilon_w (\bv))_T+(\lambda \nabla_w \cdot \bu, \nabla_w \cdot \bv)_T\\=&\sum_{T\in {\cal T}_h} (2\mu Q_{r_1} \epsilon (\bu), \epsilon_w \bv)_T+(\lambda Q_{r_2} (\nabla  \cdot \bu), \nabla_w \cdot \bv)_T\\
=&\sum_{T\in {\cal T}_h} (2\mu\epsilon(\bv_0),   Q_{r_1} \epsilon (\bu))_T+ \langle   2\mu(\bv_b-\bv_0), \frac{1}{2}( Q_{r_1} \epsilon (\bu)+Q_{r_1} \epsilon (\bu)^T) \cdot\bn\rangle_{\partial T}\\&+  (\lambda \nabla\cdot \bv_0,  Q_{r_2} (\nabla  \cdot \bu))_T+ \langle \lambda (\bv_b-\bv_0)\cdot\bn,   Q_{r_2} (\nabla  \cdot \bu)\rangle_{\partial T}\\
=&\sum_{T\in {\cal T}_h} (2\mu\epsilon(\bv_0),     \epsilon (\bu))_T+ \langle   2\mu(\bv_b-\bv_0),  Q_{r_1}  \epsilon (\bu) \cdot\bn\rangle_{\partial T}\\&+   (\lambda \nabla\cdot \bv_0,     \nabla  \cdot \bu)_T+ \langle \lambda (\bv_b-\bv_0)\cdot\bn,    Q_{r_2} (\nabla  \cdot \bu)\rangle_{\partial T}\\
=&\sum_{T\in {\cal T}_h}(\bf,\bv_0) +\langle 2\mu \epsilon(\bu)\cdot\bn, \bv_0\rangle_{\partial T}+\langle \lambda \nabla\cdot\bu, \bv_0\cdot\bn\rangle_{\partial T}\\&+\langle   \bv_b-\bv_0, 2\mu Q_{r_1}  \epsilon (\bu) \cdot\bn\rangle_{\partial T} +\langle (\bv_b-\bv_0)\cdot\bn, \lambda   Q_{r_2} (\nabla  \cdot \bu)\rangle_{\partial T}\\
=&\sum_{T\in {\cal T}_h}(\bf,\bv_0)  +\langle   \bv_b-\bv_0, 2\mu( Q_{r_1}-I)  \epsilon (\bu) \cdot\bn\rangle_{\partial T} \\&+\langle (\bv_b-\bv_0)\cdot\bn, \lambda   (Q_{r_2} -I)(\nabla  \cdot \bu)\rangle_{\partial T},
\end{split}
\end{equation}
where we used \eqref{model}, the usual integration by parts, and the fact that $\sum_{T\in {\cal T}_h} \langle 2\mu \epsilon(\bu)\cdot\bn, \bv_b\rangle_{\partial T} = \langle 2\mu \epsilon(\bu)\cdot\bn, \bv_b\rangle_{\partial \Omega}=0$ and $\sum_{T\in {\cal T}_h}\langle   \lambda \nabla\cdot\bu, \bv_b\cdot\bn\rangle_{\partial T} = \langle \lambda \nabla\cdot\bu, \bv_b\cdot\bn\rangle_{\partial \Omega}=0$  since $\bv_b=0$ on $\partial \Omega$.  

Finally, subtracting \eqref{WG} from \eqref{54}   completes the proof of the lemma.
\end{proof}

\section{Error Estimates}

\begin{lemma}\cite{wy3655}
Let ${\cal T}_h$ be a finite element partition of the domain $\Omega$ that satisfies the shape regular assumption  outlined in  \cite{wy3655}. For any $0\leq s \leq 1$, $0\leq m \leq k+1$, and $0\leq n \leq k$, there holds
\begin{eqnarray}\label{error1}
 \sum_{T\in {\cal T}_h}h_T^{2s}\|\epsilon (\bu)- Q_{r_1} \epsilon( \bu)\|^2_{s,T}&\leq& C  h^{2m-2}\|\bu\|^2_{m},\\
 \label{error3} \sum_{T\in {\cal T}_h}h_T^{2s}\|\nabla\cdot \bu- Q_{r_2} \nabla\cdot \bu\|^2_{s,T}&\leq& C  h^{2m-2}\|\bu\|^2_{m},\\
\label{error2}
\sum_{T\in {\cal T}_h}h_T^{2s}\|\bu- Q _0\bu\|^2_{s,T}&\leq& C h^{2n+2}\|\bu\|^2_{n+1}.
\end{eqnarray}
 \end{lemma}

\begin{lemma}
Assume  the exact solution $\bu$ of the elasticity problem \eqref{model} is sufficiently regular, such that $\bu\in [H^{k+1} (\Omega)]^d$. Then, there exists a constant $C$ such that the following estimate holds:
\begin{equation}\label{erroresti1}
\3bar \bu-Q_h\bu \3bar \leq Ch^k\|\bu\|_{k+1}.
\end{equation}
\end{lemma}
\begin{proof}
Recall that $\mu$ and $\lambda$ are the Lam$\acute{e}$ constants.  Using \eqref{2.5new}, the Cauchy-Schwarz inequality, the trace inequalities \eqref{tracein}-\eqref{trace}, and the estimate \eqref{error2} with $n=k$ and $s=0, 1$,  we obtain:
\begin{equation*}
\begin{split} 
&\quad\sum_{T\in {\cal T}_h}(2\mu\epsilon_w(\bu-Q_h\bu), \bvarphi)_T \\
&=\sum_{T\in {\cal T}_h} (2\mu\epsilon(\bu-Q_0\bu),  \bvarphi)_T+\langle 2\mu(Q_0\bu-Q_b\bu), \frac{1}{2}(\bvarphi+\bvarphi^T)\cdot\bn\rangle_{\partial T}\\
&\leq \Big(\sum_{T\in {\cal T}_h}\|2\mu\epsilon(\bu-Q_0\bu)\|^2_T\Big)^{\frac{1}{2}} \Big(\sum_{T\in {\cal T}_h}\|\bvarphi\|_T^2\Big)^{\frac{1}{2}}\\&\quad+ \Big(\sum_{T\in {\cal T}_h} \|2\mu(Q_0\bu-Q_b\bu)\|_{\partial T} ^2\Big)^{\frac{1}{2}}\Big(\sum_{T\in {\cal T}_h} \|\bvarphi\|_{\partial T}^2\Big)^{\frac{1}{2}}\\
&\leq C\Big(\ \sum_{T\in {\cal T}_h} \|\epsilon(\bu-Q_0\bu)\|_T^2\Big)^{\frac{1}{2}}\Big(\sum_{T\in {\cal T}_h} \|\bvarphi\|_T^2\Big)^{\frac{1}{2}}\\&\quad+C\Big(\sum_{T\in {\cal T}_h}h_T^{-1} \|Q_0\bu-\bu\|_{T} ^2+h_T \|Q_0\bu-\bu\|_{1,T} ^2\Big)^{\frac{1}{2}}\Big(\sum_{T\in {\cal T}_h} h_T^{-1}\|\bvarphi\|_T^2\Big)^{\frac{1}{2}}\\
&\leq Ch^k\|\bu\|_{k+1}\Big(\sum_{T\in {\cal T}_h} \|\bvarphi\|_T^2\Big)^{\frac{1}{2}},
\end{split}
\end{equation*}
 for any $\bvarphi\in [P_{r_1}(T)]^{d\times d}$. 
 
Letting $\bvarphi= \epsilon_w(\bu-Q_h\bu)$ yields 
\begin{equation}\label{n1}
\begin{split}
   &  \sum_{T\in {\cal T}_h}(2\mu\epsilon_w(\bu-Q_h\bu),  \epsilon_w(\bu-Q_h\bu))_T\\\leq & 
 Ch^k\|\bu\|_{k+1}\Big(\sum_{T\in {\cal T}_h} \|\epsilon_w(\bu-Q_h\bu)\|_T^2\Big)^{\frac{1}{2}}. 
\end{split}
\end{equation}

Using \eqref{disdivnew},the Cauchy-Schwarz inequality, the trace inequalities \eqref{tracein}-\eqref{trace}, and the estimate \eqref{error2} with $n=k$ and $s=0, 1$,  we have

  \begin{equation*}
\begin{split}
&\quad\sum_{T\in {\cal T}_h}(\lambda \nabla_w \cdot(\bu-Q_h\bu), \phi)_T \\
&=\sum_{T\in {\cal T}_h} (\lambda \nabla\cdot(\bu-Q_0\bu),  \phi)_T+\langle \lambda(Q_0\bu-Q_b\bu)\cdot\bn,  \phi\rangle_{\partial T}\\
&\leq \Big(\sum_{T\in {\cal T}_h}\|\lambda \nabla\cdot(\bu-Q_0\bu)\|^2_T\Big)^{\frac{1}{2}} \Big(\sum_{T\in {\cal T}_h}\|\phi\|_T^2\Big)^{\frac{1}{2}}\\&\quad+ \Big(\sum_{T\in {\cal T}_h} \|\lambda(Q_0\bu-Q_b\bu)\cdot\bn\|_{\partial T} ^2\Big)^{\frac{1}{2}}\Big(\sum_{T\in {\cal T}_h} \|\phi\|_{\partial T}^2\Big)^{\frac{1}{2}}\\
&\leq C\Big(\ \sum_{T\in {\cal T}_h} \|\nabla\cdot(\bu-Q_0\bu)\|_T^2\Big)^{\frac{1}{2}}\Big(\sum_{T\in {\cal T}_h} \|\phi\|_T^2\Big)^{\frac{1}{2}}\\&\quad+C\Big(\sum_{T\in {\cal T}_h}h_T^{-1} \|Q_0\bu-\bu\|_{T} ^2+h_T \|Q_0\bu-\bu\|_{1,T} ^2\Big)^{\frac{1}{2}}\Big(\sum_{T\in {\cal T}_h}h_T^{-1}\|\phi\|_T^2\Big)^{\frac{1}{2}}\\
&\leq Ch^k\|\bu\|_{k+1}\Big(\sum_{T\in {\cal T}_h} \|\phi\|_T^2\Big)^{\frac{1}{2}},
\end{split}
\end{equation*}
 for any $\phi\in  P_{r_2}(T)$. 
 
Letting $\phi= \nabla_w \cdot(\bu-Q_h\bu)$ gives 
\begin{equation}\label{n2}
\begin{split}
  & \sum_{T\in {\cal T}_h}(\lambda \nabla_w \cdot(\bu-Q_h\bu), \nabla_w \cdot(\bu-Q_h\bu))_T\\ \leq &
 Ch^k\|\bu\|_{k+1}\Big(\sum_{T\in {\cal T}_h} \|\nabla_w \cdot(\bu-Q_h\bu)\|_T^2\Big)^{\frac{1}{2}}. 
\end{split}
\end{equation}

Combining \eqref{n1} and \eqref{n2}, we conclude that $$
 \3bar\bu-Q_h\bu\3bar ^2\leq 
 Ch^k\|\bu\|_{k+1}\Big(\Big(\sum_{T\in {\cal T}_h} \|\epsilon_w(\bu-Q_h\bu)\|_T^2\Big)^{\frac{1}{2}}+\Big(\sum_{T\in {\cal T}_h} \|\nabla_w \cdot(\bu-Q_h\bu)\|_T^2\Big)^{\frac{1}{2}}\Big).
 $$
 This completes the proof of the lemma.
\end{proof}

\begin{theorem}
Assume  the exact solution $\bu$ of the elasticity problem \eqref{model} is sufficiently regular such that $\bu\in [H^{k+1} (\Omega)]^d$. Then, there exists a constant $C$, such that the following error estimate holds: 
\begin{equation}\label{trinorm}
\3bar \bu-\bu_h\3bar \leq Ch^k\|\bu\|_{k+1}.
\end{equation}
\end{theorem}
\begin{proof}
For   the right-hand side of the error equation \eqref{erroreqn}, using the Cauchy-Schwarz inequality, the trace inequality \eqref{tracein},  the estimates \eqref{error1}-\eqref{error3} with $m=k+1$ and $s=0,1$, and \eqref{normeq}, we have   
\begin{equation}\label{erroreqn1}
\begin{split}
& \Big|\sum_{T\in {\cal T}_h}\langle   \bv_b-\bv_0, 2\mu( Q_{r_1}-I)  \epsilon (\bu) \cdot\bn\rangle_{\partial T} \\&+\langle (\bv_b-\bv_0)\cdot\bn, \lambda   (Q_{r_2} -I)(\nabla  \cdot \bu)\rangle_{\partial T}\Big|\\
\leq & C\Big(\sum_{T\in {\cal T}_h}\|2\mu( Q_{r_1}-I)  \epsilon (\bu) \cdot\bn\|^2_T+h_T^2\|2\mu( Q_{r_1}-I)  \epsilon (\bu) \cdot\bn\|^2_{1,T}\Big)^{\frac{1}{2}}  \\& \cdot\Big(\sum_{T\in {\cal T}_h}h_T^{-1}\|\bv_b-\bv_0\|^2_{\partial T}\Big)^{\frac{1}{2}}\\
&+C\Big(\sum_{T\in {\cal T}_h}\|\lambda(Q_{r_2} -I)(\nabla  \cdot \bu)\|^2_T+h_T^2\|\lambda(Q_{r_2} -I)(\nabla  \cdot \bu)\|^2_{1,T}\Big)^{\frac{1}{2}}  \\&\cdot\Big(\sum_{T\in {\cal T}_h}h_T^{-1}\|\bv_b-\bv_0\|^2_{\partial T}\Big)^{\frac{1}{2}}\\
\leq & Ch^k\|\bu\|_{k+1} \| \bv\|_{1,h}\\
\leq & Ch^k\|\bu\|_{k+1} \3bar \bv\3bar.
\end{split}
\end{equation}
Substituting \eqref{erroreqn1}  into \eqref{erroreqn}  gives
\begin{equation}\label{err}
\sum_{T\in {\cal T}_h}(2\mu\epsilon_w (\be_h), \epsilon_w (\bv))_T+(\lambda \nabla_w \cdot \be_h, \nabla_w \cdot \bv)_T\leq   Ch^k\|\bu\|_{k+1} \3bar  \bv\3bar.
\end{equation} 

Now, applying the  Cauchy-Schwarz inequality and letting $\bv=Q_h\bu-\bu_h$ in \eqref{err}, along with  the estimate \eqref{erroresti1}, we obtain
\begin{equation*}
\begin{split}
& \3bar \bu-\bu_h\3bar^2\\=&\sum_{T\in {\cal T}_h}(2\mu\epsilon_w (\bu-\bu_h), \epsilon_w (\bu-Q_h\bu))_T+(2\mu\epsilon_w (\bu-\bu_h), \epsilon_w (Q_h\bu-\bu_h))_T\\
&+(\lambda\nabla_w \cdot(\bu-\bu_h), \nabla_w \cdot (\bu-Q_h\bu))_T+(\lambda\nabla_w \cdot (\bu-\bu_h), \nabla_w \cdot (Q_h\bu-\bu_h))_T\\
\leq &\Big(\sum_{T\in {\cal T}_h}2\mu\|\epsilon_w (\bu-\bu_h)\|^2_T\Big)^{\frac{1}{2}} \Big(\sum_{T\in {\cal T}_h}2\mu\| \epsilon_w(\bu-Q_h\bu)\|^2_T\Big)^{\frac{1}{2}} \\
&+ \Big(\sum_{T\in {\cal T}_h}\lambda\|\nabla_w \cdot(\bu-\bu_h)\|^2_T\Big)^{\frac{1}{2}} \Big(\sum_{T\in {\cal T}_h}\lambda\| \nabla_w \cdot(\bu-Q_h\bu)\|^2_T\Big)^{\frac{1}{2}} \\& +\sum_{T\in {\cal T}_h} (2\mu\epsilon_w (\bu-\bu_h), \epsilon_w (Q_h\bu-\bu_h))_T+(\lambda\nabla_w \cdot (\bu-\bu_h), \nabla_w \cdot (Q_h\bu-\bu_h))_T\\
\leq &\3bar \bu-\bu_h \3bar  \3bar \bu-Q_h\bu \3bar+ Ch^k\|\bu\|_{k+1} \3bar Q_h\bu-\bu_h\3bar\\
\leq &\3bar \bu-\bu_h  \3bar  h^k\|\bu\|_{k+1} + Ch^k\|\bu\|_{k+1}  (\3bar Q_h\bu-\bu\3bar+\3bar \bu-\bu_h \3bar)  \\
\leq &\3bar \bu-\bu_h  \3bar  h^k\|\bu\|_{k+1} + Ch^k\|\bu\|_{k+1}   h^k\|\bu\|_{k+1}+Ch^k\|\bu\|_{k+1} \3bar \bu-\bu_h\3bar.
\end{split}
\end{equation*}
This simplifies to 
\begin{equation*}
\begin{split}
 \3bar \bu-\bu_h\3bar  \leq Ch^k\|\bu\|_{k+1}.
\end{split}
\end{equation*} 

Thus, the proof of the theorem is complete.
\end{proof}

\section{Error Estimates in $L^2$ Norm} 
The standard duality argument is applied to derive the $L^2$ error estimate. Recall that the error function  $\be_h=\bu-\bu_h=\{\be_0, \be_b\}$  represents the difference between the exact solution $\bu$ and the computed solution $\bu_h$.  Additionally, we define $  \bzeta_h =Q_h\bu - \bu_h=\{  \bzeta_0,   \bzeta_b\}$. 

The dual problem for the elasticity problem \eqref{model}  seeks $\bw$  such that 
\begin{equation}\label{dual}
\begin{split}
    -\nabla \cdot (2\mu \epsilon(\bw)+\lambda(\nabla\cdot \bw)\bI)=&\bzeta_0,\qquad \text{in} \quad \Omega,\\
    \bw=&0, \qquad \text{on} \quad \partial\Omega.
\end{split}
\end{equation}
Assume that  the dual problem \eqref{dual}  satisfies the following regularity estimate:
 \begin{equation}\label{regu2}
  \|\bw\|_{1+\alpha} \leq C\|  \bzeta_0\|,
 \end{equation}
 where $\frac{1}{2}<\alpha\leq 1$.
 \begin{theorem}
Let  $\bu$  be the exact solution of the elasticity problem \eqref{model}, assuming sufficient regularity such that $\bu\in [H^{k+1} (\Omega)]^d$.  Let $\bu_h\in V_h$ be the numerical solution of the WG Algorithm \ref{PDWG1}.  If the regularity assumption  \eqref{regu2} holds for the dual problem \eqref{dual}, then there exists a constant $C$ such that 
\begin{equation*}
\|\be_0\|\leq Ch^{k+\alpha}\|\bu\|_{k+1}.
\end{equation*}
 \end{theorem}
 
 \begin{proof}
 Testing \eqref{dual} with $\bzeta_0$ gives
 \begin{equation}\label{e1}
 \begin{split}
 \|\bzeta_0\|^2 =&( -\nabla \cdot (2\mu \epsilon(\bw)+\lambda(\nabla\cdot \bw)\bI),   \bzeta_0) 
  \\=& \sum_{T\in {\cal T}_h} (2\mu \epsilon(\bw)+\lambda (\nabla\cdot \bw)\bI, \nabla \bzeta_0)_T-\langle  2\mu \epsilon(\bw)\cdot\bn,   \bzeta_0  \rangle_{\partial T}\\&-\langle   \lambda  \nabla\cdot \bw,   \bzeta_0 \cdot\bn \rangle_{\partial T}\\
   =&  \sum_{T\in {\cal T}_h} (2\mu \epsilon(\bw),\epsilon( \bzeta_0))_T+( \lambda \nabla\cdot \bw, \nabla\cdot \bzeta_0)_T\\&-\langle  2\mu \epsilon(\bw)\cdot\bn,   \bzeta_0   -\bzeta_b\rangle_{\partial T}-\langle   \lambda  \nabla\cdot \bw,   (\bzeta_0-\bzeta_b) \cdot\bn \rangle_{\partial T}, 
 \end{split}
 \end{equation}
where  integration by parts has been used along with the fact that 
$\sum_{T\in {\cal T}_h} \langle    2\mu \epsilon(\bw)\cdot\bn,   \bzeta_b\rangle_{\partial T}=\langle 2\mu \epsilon(\bw)\cdot\bn,   \bzeta_b \rangle_{\partial \Omega}=0$ and $\sum_{T\in {\cal T}_h} \langle   \lambda  \nabla\cdot \bw,  \bzeta_b  \cdot\bn\rangle_{\partial T}=\langle   \lambda  \nabla\cdot \bw,  \bzeta_b  \cdot\bn \rangle_{\partial \Omega}=0$ since $  \bzeta_b=Q_b\bu-\bu_b=0$ on $\partial\Omega$.

 Letting  $\bu=\bw$ and $\bv=  \bzeta_h$ in \eqref{54}, we get:
\begin{equation*} 
\begin{split}
&\sum_{T\in {\cal T}_h} (2\mu\epsilon_w (\bw), \epsilon_w (\bzeta_h))_T+(\lambda \nabla_w \cdot \bw, \nabla_w \cdot \bzeta_h))_T\\
= &\sum_{T\in {\cal T}_h} (2\mu\epsilon(\bzeta_0),     \epsilon (\bw))_T+ \langle  2\mu (\bzeta_b-\bzeta_0),  Q_{r_1}  \epsilon (\bw) \cdot\bn\rangle_{\partial T}\\&+  (\lambda\nabla\cdot \bzeta_0,     \nabla  \cdot \bw)_T+ \langle \lambda(\bzeta_b-\bzeta_0)\cdot\bn,    Q_{r_2} (\nabla  \cdot \bw)\rangle_{\partial T}.
\end{split}
\end{equation*}  
This can be rewritten as: 
\begin{equation*} 
\begin{split}
&\sum_{T\in {\cal T}_h} (2\mu\epsilon(\bzeta_0),     \epsilon (\bw))_T+\lambda(\nabla\cdot \bzeta_0,     \nabla  \cdot \bw)_T\\=&\sum_{T\in {\cal T}_h} (2\mu\epsilon_w (\bw), \epsilon_w (\bzeta_h))_T+(\lambda \nabla_w \cdot \bw, \nabla_w \cdot \bzeta_h)_T\\&- \langle  2\mu (\bzeta_b-\bzeta_0),  Q_{r_1}  \epsilon (\bw) \cdot\bn\rangle_{\partial T}- \langle \lambda(\bzeta_b-\bzeta_0)\cdot\bn,    Q_{r_2} (\nabla  \cdot \bw)\rangle_{\partial T}.
\end{split}
\end{equation*} 
Substituting this equation into \eqref{e1} and using \eqref{erroreqn}, we obtain:
\begin{equation}\label{e2}
 \begin{split}
& \|  \bzeta_0\|^2  \\
  = &\sum_{T\in {\cal T}_h}(2\mu\epsilon_w (\bw), \epsilon_w (\bzeta_h))_T+(\lambda \nabla_w \cdot \bw, \nabla_w \cdot \bzeta_h)_T\\&- \langle  2\mu (\bzeta_b-\bzeta_0),  (Q_{r_1}-I)  \epsilon (\bw) \cdot\bn\rangle_{\partial T}\\&-\lambda\langle (\bzeta_b-\bzeta_0)\cdot\bn,   ( Q_{r_2}-I) (\nabla  \cdot \bw)\rangle_{\partial T}\\
  =& \sum_{T\in {\cal T}_h}(2\mu\epsilon_w (\bw), \epsilon_w (\be_h))_T+(\lambda \nabla_w \cdot \bw, \nabla_w \cdot \be_h)_T\\&+(2\mu\epsilon_w (\bw), \epsilon_w (Q_h\bu-\bu))_T+(\lambda \nabla_w \cdot \bw, \nabla_w \cdot (Q_h\bu-\bu))_T\\&-\ell(\bw,   \bzeta_h)\\ 
  =& \sum_{T\in {\cal T}_h}(2\mu\epsilon_w (Q_h\bw), \epsilon_w (\be_h))_T+(2\mu\epsilon_w (\bw-Q_h\bw), \epsilon_w (\be_h))_T\\&+(\lambda \nabla_w \cdot Q_h\bw, \nabla_w \cdot \be_h)_T+(\lambda \nabla_w \cdot (\bw-Q_h\bw), \nabla_w \cdot \be_h)_T\\&+(2\mu\epsilon_w (\bw), \epsilon_w (Q_h\bu-\bu))_T+(\lambda \nabla_w \cdot \bw, \nabla_w \cdot (Q_h\bu-\bu))_T\\&-\ell(\bw,   \bzeta_h)\\ 
  =&\ell(\bu, Q_h\bw) + \sum_{T\in {\cal T}_h} (2\mu\epsilon_w (\bw-Q_h\bw), \epsilon_w (\be_h))_T\\&+(\lambda \nabla_w \cdot (\bw-Q_h\bw), \nabla_w \cdot \be_h)_T +(2\mu\epsilon_w (\bw), \epsilon_w (Q_h\bu-\bu))_T\\&+(\lambda \nabla_w \cdot \bw, \nabla_w \cdot (Q_h\bu-\bu))_T-\ell(\bw,   \bzeta_h)\\
  =&\sum_{i=1}^6 J_i.
 \end{split}
 \end{equation}
 
We will estimate the six terms $J_i$ for $i=1,\cdots, 6$ on the last line of \eqref{e2} individually.

Regarding to $J_1$, using the Cauchy-Schwarz inequality, the trace inequality \eqref{tracein},   the estimate \eqref{error1} with $m=k+1$ and $s=0, 1$, and the estimate \eqref{error2} with $n=\alpha$, we have:
\begin{equation*}\label{ee1}
\begin{split}
J_1=&\ell(\bu, Q_h\bw)\\
\leq &|\sum_{T\in {\cal T}_h}  \langle   Q_b\bw-Q_0\bw, 2\mu( Q_{r_1}-I)  \epsilon (\bu) \cdot\bn\rangle_{\partial T}| \\&+|\sum_{T\in {\cal T}_h} \langle (Q_b\bw-Q_0\bw)\cdot\bn, \lambda   (Q_{r_2} -I)(\nabla  \cdot \bu)\rangle_{\partial T}|\\
\leq& \Big(\sum_{T\in {\cal T}_h}\|Q_b\bw-Q_0\bw\|_{\partial T}^2\Big)^{\frac{1}{2}} \Big(\sum_{T\in {\cal T}_h}\|2\mu( Q_{r_1}-I)  \epsilon (\bu) \cdot\bn\|_{\partial T}^2\Big)^{\frac{1}{2}} \\
&+\Big(\sum_{T\in {\cal T}_h}\| (Q_b\bw-Q_0\bw)\cdot\bn\|_{\partial T}^2\Big)^{\frac{1}{2}} \Big(\sum_{T\in {\cal T}_h}\|\lambda(Q_{r_2} -I)(\nabla  \cdot \bu)\|_{\partial T}^2\Big)^{\frac{1}{2}} \\
\leq& \Big(\sum_{T\in {\cal T}_h}h_T^{-1}\| \bw-Q_0\bw\|_{T}^2+h_T \| \bw-Q_0\bw\|_{1, T}^2\Big)^{\frac{1}{2}} \\&\cdot\Big(\sum_{T\in {\cal T}_h}h_T^{-1}\|2\mu( Q_{r_1}-I)  \epsilon (\bu) \cdot\bn\|_{T}^2+h_T \|2\mu( Q_{r_1}-I)  \epsilon (\bu) \cdot\bn\|_{1,T}^2\Big)^{\frac{1}{2}} \\
&+\Big(\sum_{T\in {\cal T}_h}h_T^{-1}\|(\bw-Q_0\bw)\cdot\bn\|_{T}^2+h_T \|(\bw-Q_0\bw)\cdot\bn\|_{1, T}^2\Big)^{\frac{1}{2}} \\&\cdot\Big(\sum_{T\in {\cal T}_h}h_T^{-1}\|\lambda(Q_{r_2} -I)(\nabla  \cdot \bu)\|_{ T}^2+h_T \|\lambda(Q_{r_2} -I)(\nabla  \cdot \bu)\|_{1,  T}^2\Big)^{\frac{1}{2}} 
\\
\leq &Ch^{-1}h^k\|\bu\|_{k+1}h^{1+\alpha}\|\bw\|_{1+\alpha}\\
\leq & Ch^{k+\alpha}\|\bu\|_{k+1}\|\bw\|_{1+\alpha}.
\end{split}
\end{equation*}

For $J_2$, using the Cauchy-Schwarz inequality, \eqref{erroresti1} with $k=\alpha$ and \eqref{trinorm}, we obtain:
\begin{equation*}\label{ee2}
\begin{split}
J_2 \leq  \3bar \bw-Q_h\bw\3bar \3bar \be_h\3bar\leq Ch^{k}\|\bu\|_{k+1}h^{\alpha}\|\bw\|_{1+\alpha}\leq Ch^{k+\alpha}\|\bu\|_{k+1}\|\bw\|_{1+\alpha}.
\end{split}
\end{equation*}

For $J_3$, using the Cauchy-Schwarz inequality, \eqref{erroresti1} with $k=\alpha$,  and \eqref{trinorm}, we find:
\begin{equation*}\label{ee2}
\begin{split}
J_3 \leq C\3bar \bw-Q_h\bw\3bar \3bar \be_h\3bar\leq Ch^{k}\|\bu\|_{k+1}h^{\alpha}\|\bw\|_{1+\alpha}\leq Ch^{k+\alpha}\|\bu\|_{k+1}\|\bw\|_{1+\alpha}.
\end{split}
\end{equation*}

For $J_4$, let $Q^0$ denote the $L^2$-projection onto $[P_0(T)]$. Using \eqref{2.5}, we get:
  \begin{equation}\label{ee}
 \begin{split}
 &(2\mu Q^0(\epsilon_w (\bw)), \epsilon_w (Q_h\bu-\bu))_T\\ =& - (2\mu(Q_0\bu-\bu), \nabla \cdot(Q^0(\epsilon_w (\bw)))_T \\&+ \langle 2\mu(Q_b\bu-\bu), Q^0(\epsilon_w (\bw))\cdot\bn\rangle_{\partial T}=0.
 \end{split}
 \end{equation}
  Using \eqref{ee}, Cauchy-Schwarz inequality, \eqref{pro} and \eqref{erroresti1}, gives 
  \begin{equation*}
  \begin{split}
  J_4\leq &|\sum_{T\in {\cal T}_h}(2\mu\epsilon_w (\bw), \epsilon_w (Q_h\bu-\bu))_T|
  \\
  =&|\sum_{T\in {\cal T}_h} ( 2\mu(\epsilon_w (\bw)-  Q^0(\epsilon_w (\bw))), \epsilon_w (Q_h\bu-\bu))_T|\\
  =&|\sum_{T\in {\cal T}_h} ( 2\mu(Q_{r_1}\epsilon (\bw)-  Q^0(Q_{r_1}\epsilon (\bw))), \epsilon_w (Q_h\bu-\bu))_T|\\
  \leq & \Big(\sum_{T\in {\cal T}_h}\|Q_{r_1}\epsilon (\bw)-  Q^0(Q_{r_1}\epsilon (\bw))\|_T^2\Big)^{\frac{1}{2}} \3bar Q_h\bu-\bu \3bar\\
  \leq & Ch^k\|\bu\|_{k+1}h^{\alpha}\|\bw\|_{1+\alpha}\\
  \leq & Ch^{k+\alpha}\|\bu\|_{k+1}\|\bw\|_{1+\alpha}.
  \end{split}
  \end{equation*}

Similar to the estimate for $J_4$, we obtain:
$$
J_5\leq Ch^{k+\alpha}\|\bu\|_{k+1}\|\bw\|_{1+\alpha}.
$$

For $J_6$, using the Cauchy-Schwarz inequality, the trace inequality \eqref{tracein},  Lemma \ref{normeqva}, and  the estimates \eqref{error1}-\eqref{error3} with $m= 1+\alpha$ and $s=0, 1$, along with   \eqref{erroresti1} and \eqref{trinorm}, we obtain:

\begin{equation*}\label{ee4}
\begin{split}
J_6=&\ell(\bw,   \bzeta_h)
\\
\leq &\Big|\sum_{T\in {\cal T}_h}\langle    \bzeta_b- \bzeta_0, 2\mu( Q_{r_1}-I)  \epsilon (\bw) \cdot\bn\rangle_{\partial T}\Big|\\&+\Big|\sum_{T\in {\cal T}_h}\langle ( \bzeta_b- \bzeta_0)\cdot\bn, \lambda   (Q_{r_2} -I)(\nabla  \cdot \bw)\rangle_{\partial T}\Big| \\
\leq& \Big(\sum_{T\in {\cal T}_h}\| \bzeta_b- \bzeta_0\|_{\partial T}^2\Big)^{\frac{1}{2}} \Big(\sum_{T\in {\cal T}_h}\|2\mu( Q_{r_1}-I)  \epsilon (\bw) \cdot\bn\|_{\partial T}^2\Big)^{\frac{1}{2}}   \\
&+\Big(\sum_{T\in {\cal T}_h}\|( \bzeta_b- \bzeta_0)\cdot\bn\|_{\partial T}^2\Big)^{\frac{1}{2}} \Big(\sum_{T\in {\cal T}_h}\|\lambda   (Q_{r_2} -I)(\nabla  \cdot \bw)\|_{\partial T}^2\Big)^{\frac{1}{2}}   \\
\leq& \Big(\sum_{T\in {\cal T}_h}h_T^{-1}\| \bzeta_b- \bzeta_0\|_{\partial T}^2\Big)^{\frac{1}{2}}\\&\cdot \Big(\sum_{T\in {\cal T}_h}\|2\mu( Q_{r_1}-I)  \epsilon (\bw) \cdot\bn\|_{ T}^2+h_T^2\|2\mu( Q_{r_1}-I)  \epsilon (\bw) \cdot\bn\|_{1, T}^2\Big)^{\frac{1}{2}}   \\
&+\Big(\sum_{T\in {\cal T}_h}h_T^{-1}\|( \bzeta_b- \bzeta_0)\cdot\bn\|_{\partial T}^2\Big)^{\frac{1}{2}} \\&\cdot\Big(\sum_{T\in {\cal T}_h}\|\lambda   (Q_{r_2} -I)(\nabla  \cdot \bw)\|_{T}^2+h_T^2\|\lambda   (Q_{r_2} -I)(\nabla  \cdot \bw)\|_{1, T}^2\Big)^{\frac{1}{2}}   \\
\leq& Ch^{\alpha}\|\bw\|_{1+\alpha} \|   \bzeta_h\|_{1,h}  \\
\leq& Ch^{\alpha}\|\bw\|_{1+\alpha} \3bar   \bzeta_h\3bar  \\
\leq &Ch^{\alpha}\|\bw\|_{1+\alpha} (\3bar \bu-\bu_h\3bar+\3bar \bu-Q_h\bu\3bar)  \\
\leq & Ch^{\alpha}\|\bw\|_{1+\alpha} (h^k\|\bu\|_{k+1}+h^{k}\|\bu\|_{k+1})\\
\leq &Ch^{k+\alpha}\|\bw\|_{1+\alpha}\|\bu\|_{k+1}.
\end{split}
\end{equation*}

  Substituting the estimates for $J_i$ for $i=1,\cdots, 6$ into \eqref{e2} and using \eqref{regu2} gives
$$
\|  \bzeta_0\|^2\leq Ch^{k+\alpha}\|\bw\|_{1+\alpha}\|\bu\|_{k+1}\leq Ch^{k+\alpha}  \|\bu\|_{k+1} \|  \bzeta_0\|.
$$
This gives
$$
\|  \bzeta_0\|\leq Ch^{k+\alpha} \|\bu\|_{k+1},
$$
which, using triangle inequality and \eqref{error2} with $n=k$, gives
$$
\|\be_0\|\leq \|  \bzeta_0\|+\|\bu-Q_0\bu\|\leq Ch^{k+\alpha}\|\bu\|_{k+1}. 
$$

This completes the proof of the theorem. 
\end{proof}

\section{Numerical verification}

In the first test,  we solve the linear elasticity equation \eqref{WG} with $\mu=1$ on the unit square domain $\Omega
  =(0,1)\times(0,1)$.
We choose $\b f$ and $\b g$ in  \eqref{WG} such that the exact solution is
\an{\label{e1} \b u=\p{ (x^2-2x^3+x^4)(2y-6y^2+4y^3) \\ - (y^2-2y^3+y^4)(2x-6x^2+4x^3)}.  }

     \begin{figure}[H] \setlength\unitlength{1pt}\begin{center}
    \begin{picture}(300,100)(0,0)
     \def\mc{\begin{picture}(90,90)(0,0)
       \put(0,0){\line(1,0){90}} \put(0,90){\line(1,0){90}}
      \put(0,0){\line(0,1){90}}  \put(90,0){\line(0,1){90}}  \put(0,0){\line(1,1){90}}  
      \end{picture}}

    \put(0,0){\mc} \put(5,92){Grid 1:} \put(110,92){Grid 2:}\put(215,92){Grid 3:}
      \put(105,0){\setlength\unitlength{0.5pt}\begin{picture}(90,90)(0,0)
    \put(0,0){\mc}\put(90,0){\mc}\put(0,90){\mc}\put(90,90){\mc}\end{picture}}
      \put(210,0){\setlength\unitlength{0.25pt}\begin{picture}(90,90)(0,0)
    \multiput(0,0)(90,0){4}{\multiput(0,0)(0,90){4}{\mc}} \end{picture}}
    \end{picture}
 \caption{The first three grids for the computation in
    Tables \ref{t-1}--\ref{t-2}. }\label{grid1} 
    \end{center} \end{figure}
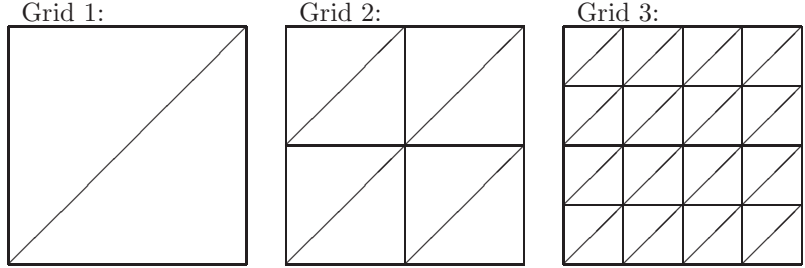

We compute the finite element solutions for the solution \eqref{e1} on uniform 
       triangular grids shown in Figure \ref{grid1} by 
  the  $P_k$ WG finite elements, defined in \eqref{Vk}--\eqref{Vh0}, 
      for $k=1,2,3$ and $4$.
The results are listed in Tables \ref{t-1}--\ref{t-2}. 
The optimal order of convergence is achieved for all solutions in all norms. 
Further, the method is shown pressure robust as the convergence is independent of
  $\lambda$ in \eqref{WG}.

\begin{table}[H]
  \centering  \renewcommand{\arraystretch}{1.1}
  \caption{The error and the computed order of convergence for the solution \eqref{e1} on Figure \ref{grid1} meshes: \ (a) $P_1$ WG, $\lambda=1$;   \ (b) $P_1$ WG, $\lambda=10^7$; \ (c) $P_2$ WG, $\lambda=1$;   \ (d) $P_2$ WG, $\lambda=10^7$.  }
  \label{t-1}
\begin{tabular}{cc|cc|cc}
\hline
Method & Grid &  $\|\b u-\b u_0\|_0$  & $O(h^r)$ & $\3bar\b u-\b u_h\3bar$  & $O(h^r)$   \\
\hline
\multirow{3}{1mm}{(a)} & 
 5&    0.824E-04 &  2.0&    0.944E-02 &  1.0\\
& 6&    0.207E-04 &  2.0&    0.473E-02 &  1.0\\
& 7&    0.517E-05 &  2.0&    0.237E-02 &  1.0\\
\hline
\multirow{3}{1mm}{(b)} & 
 5&    0.101E-03 &  2.0&    0.102E-01 &  1.0\\
& 6&    0.254E-04 &  2.0&    0.513E-02 &  1.0\\
& 7&    0.632E-05 &  2.0&    0.256E-02 &  1.0\\
\hline
\multirow{3}{1mm}{(c)} & 
 5&    0.144E-05 &  3.0&    0.416E-03 &  2.0\\
&6&    0.180E-06 &  3.0&    0.104E-03 &  2.0\\
&7&    0.224E-07 &  3.0&    0.261E-04 &  2.0\\
\hline
\multirow{3}{1mm}{(d)} & 
 4&    0.129E-04 &  3.0&    0.169E-02 &  1.9\\
&5&    0.160E-05 &  3.0&    0.425E-03 &  2.0\\
&6&    0.212E-06 &  2.9&    0.107E-03 &  2.0\\
\hline
    \end{tabular}%
\end{table}%

\begin{table}[H]
  \centering  \renewcommand{\arraystretch}{1.1}
  \caption{The error and the computed order of convergence for the solution \eqref{e1} on Figure \ref{grid1} meshes: \ (a) $P_3$ WG, $\lambda=1$;   \ (b) $P_3$ WG, $\lambda=10^7$;  \ (c) $P_4$ WG, $\lambda=1$;   \ (d) $P_4$ WG, $\lambda=10^7$.  }
  \label{t-2}
\begin{tabular}{cc|cc|cc}
\hline
Method & Grid &  $\|\b u-\b u_0\|_0$  & $O(h^r)$ & $\3bar\b u-\b u_h\3bar$  & $O(h^r)$   \\
\hline
\multirow{3}{1mm}{(a)} & 
 4&    0.427E-06 &  4.0&    0.102E-03 &  3.0\\
& 5&    0.267E-07 &  4.0&    0.128E-04 &  3.0\\
& 6&    0.167E-08 &  4.0&    0.160E-05 &  3.0\\
\hline
\multirow{3}{1mm}{(b)} & 
 4&    0.457E-06 &  4.0&    0.102E-03 &  3.0\\
& 5&    0.310E-07 &  3.9&    0.127E-04 &  3.0\\
& 6&    0.135E-03 &  0.0&    0.253E-02 &  0.0\\
\hline
\multirow{3}{1mm}{(c)} & 
 3&    0.549E-06 &  5.0&    0.968E-04 &  4.0\\
& 4&    0.169E-07 &  5.0&    0.601E-05 &  4.0\\
& 5&    0.523E-09 &  5.0&    0.375E-06 &  4.0\\
\hline
\multirow{3}{1mm}{(d)} & 
 3&    0.568E-06 &  5.0&    0.964E-04 &  4.0\\
& 4&    0.173E-07 &  5.0&    0.599E-05 &  4.0\\
& 5&    0.567E-09 &  4.9&    0.375E-06 &  4.0\\
\hline
    \end{tabular}%
\end{table}%

We next compute the finite element solutions for \eqref{e1} on non-convex polygonal 
        grids shown in Figure \ref{grid2} by 
  the  $P_k$ WG finite elements, defined in \eqref{Vk}--\eqref{Vh0}, 
      for $k=1,2,3$ and $4$.
The results are listed in Tables \ref{t-3}--\ref{t-4}. 
The optimal order of convergence is achieved for all solutions in all norms. 
And the convergence is independent of the impressible 
   parameter $\lambda$ in \eqref{WG}.

     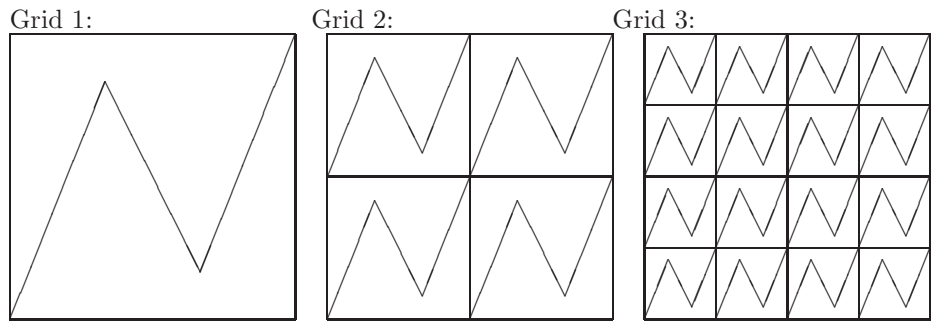
\begin{figure}[H] \setlength\unitlength{1.2pt}\begin{center}
    \begin{picture}(290,98)(0,0)
     \def\mc{\begin{picture}(90,90)(0,0)
       \put(0,0){\line(1,0){90}} \put(0,90){\line(1,0){90}}
      \put(0,0){\line(0,1){90}}  \put(90,0){\line(0,1){90}} 
      
       \put(30,75){\line(1,-2){30}}  
       \put(0,0){\line(2,5){30}}         \put(90,90){\line(-2,-5){30}}  
      \end{picture}}

    \put(0,0){\mc} \put(0,92){Grid 1:} \put(95,92){Grid 2:}  \put(190,92){Grid 3:}
      \put(100,0){\setlength\unitlength{0.6pt}\begin{picture}(90,90)(0,0)
    \put(0,0){\mc}\put(90,0){\mc}\put(0,90){\mc}\put(90,90){\mc}\end{picture}}
     \put(200,0){\setlength\unitlength{0.3pt}\begin{picture}(90,90)(0,0)
      \multiput(0,0)(90,0){4}{\multiput(0,0)(0,90){4}{\mc}} \end{picture}}
    \end{picture}
 \caption{The first three non-convex polygonal grids for the computation in
    Tables \ref{t-3}--\ref{t-4}. }\label{grid2} 
    \end{center} \end{figure}

\begin{table}[H]
  \centering  \renewcommand{\arraystretch}{1.1}
  \caption{The error and the computed order of convergence for the solution \eqref{e1} on Figure \ref{grid2}, non-convex polygonal meshes: 
  \ (a) $P_1$ WG, $\lambda=1$;   \ (b) $P_1$ WG, $\lambda=10^5$;  \ (c) $P_2$ WG, $\lambda=1$;   \ (d) $P_2$ WG, $\lambda=10^5$.  }
  \label{t-3}
\begin{tabular}{cc|cc|cc}
\hline
Method & Grid &  $\|\b u-\b u_0\|_0$  & $O(h^r)$ & $\3bar\b u-\b u_h\3bar$  & $O(h^r)$   \\
\hline
\multirow{3}{1mm}{(a)} & 
 5&    0.393E-03 &  1.8&    0.332E-01 &  1.0\\
&  6&    0.104E-03 &  1.9&    0.166E-01 &  1.0\\
&  7&    0.265E-04 &  2.0&    0.829E-02 &  1.0\\
\hline
\multirow{3}{1mm}{(b)} & 
 5&    0.320E-03 &  1.8&    0.339E-01 &  1.0\\
&  6&    0.831E-04 &  1.9&    0.170E-01 &  1.0\\
&  7&    0.210E-04 &  2.0&    0.850E-02 &  1.0\\
\hline
\multirow{3}{1mm}{(c)} & 
 4&    0.116E-03 &  2.9&    0.241E-01 &  1.9\\
&  5&    0.147E-04 &  3.0&    0.613E-02 &  2.0\\
&  6&    0.184E-05 &  3.0&    0.154E-02 &  2.0\\
\hline
\multirow{3}{1mm}{(d)} & 
 4&    0.116E-03 &  2.9&    0.240E-01 &  1.9\\
&  5&    0.146E-04 &  3.0&    0.612E-02 &  2.0\\
&  6&    0.183E-05 &  3.0&    0.154E-02 &  2.0\\
\hline
    \end{tabular}%
\end{table}%

\begin{table}[H]
  \centering  \renewcommand{\arraystretch}{1.1}
  \caption{The error and the computed order of convergence for the solution \eqref{e1} on Figure \ref{grid2}, non-convex polygonal meshes: 
  \ (a) $P_3$ WG, $\lambda=1$;   \ (b) $P_3$ WG, $\lambda=10^5$;  \ (c) $P_4$ WG, $\lambda=1$;   \ (d) $P_4$ WG, $\lambda=10^5$.  }
  \label{t-4}
\begin{tabular}{cc|cc|cc}
\hline
Method & Grid &  $\|\b u-\b u_0\|_0$  & $O(h^r)$ & $\3bar\b u-\b u_h\3bar$  & $O(h^r)$   \\
\hline
\multirow{3}{1mm}{(a)} & 
 4&    0.203E-04 &  3.9&    0.684E-02 &  2.9\\
&  5&    0.129E-05 &  4.0&    0.877E-03 &  3.0\\
&  6&    0.811E-07 &  4.0&    0.110E-03 &  3.0\\
\hline
\multirow{3}{1mm}{(b)} & 
 3&    0.298E-03 &  3.4&    0.494E-01 &  2.3\\
&  4&    0.203E-04 &  3.9&    0.684E-02 &  2.9\\
&  5&    0.130E-05 &  4.0&    0.876E-03 &  3.0\\
\hline
\multirow{3}{1mm}{(c)} & 
 3&    0.109E-03 &  4.8&    0.254E-01 &  3.7\\
&  4&    0.353E-05 &  4.9&    0.165E-02 &  3.9\\
&  5&    0.111E-06 &  5.0&    0.104E-03 &  4.0\\
\hline
\multirow{3}{1mm}{(d)} & 
 3&    0.109E-03 &  4.8&    0.254E-01 &  3.7\\
&  4&    0.374E-05 &  4.9&    0.165E-02 &  3.9\\
&  5&    0.427E-05 &  0.0&    0.109E-03 &  3.9\\
\hline
    \end{tabular}%
\end{table}%

In the 3D numerical computation,  the domain for problem \eqref{WG}
   is the unit cube $\Omega=(0,1)\times(0,1)\times(0,1)$.
We choose an $\b f$ and an $\b g$ in \eqref{WG} so that the exact solution is
\an{\label{e3} \ad{  
  \b u &=\p{e^{y+z} \\e^{z+x}\\e^{z+x} }. }  }

We   compute the finite element solutions for \eqref{e3} on the
     tetrahedral   grids shown in Figure \ref{grid4} by 
  the  $P_k$ WG finite elements, defined in \eqref{Vk}--\eqref{Vh0}, 
      for $k=1,2,3$ and $4$.
The results are listed in Tables  \ref{t-5}--\ref{t-6}. 
The optimal order of convergence is achieved for all solutions in all norms. 
And the numerical results indicate the method is pressure robust.

\begin{figure}[H] 
\begin{center}
 \setlength\unitlength{1pt}
    \begin{picture}(320,118)(0,3)
    \put(0,0){\begin{picture}(110,110)(0,0) \put(25,102){Grid 1:}
       \multiput(0,0)(80,0){2}{\line(0,1){80}}  \multiput(0,0)(0,80){2}{\line(1,0){80}}
       \multiput(0,80)(80,0){2}{\line(1,1){20}} \multiput(0,80)(20,20){2}{\line(1,0){80}}
       \multiput(80,0)(0,80){2}{\line(1,1){20}}  \multiput(80,0)(20,20){2}{\line(0,1){80}}
    \put(80,0){\line(-1,1){80}}\put(80,0){\line(1,5){20}}\put(80,80){\line(-3,1){60}}
      \end{picture}}
    \put(110,0){\begin{picture}(110,110)(0,0)\put(25,102){Grid 2:}
       \multiput(0,0)(40,0){3}{\line(0,1){80}}  \multiput(0,0)(0,40){3}{\line(1,0){80}}
       \multiput(0,80)(40,0){3}{\line(1,1){20}} \multiput(0,80)(10,10){3}{\line(1,0){80}}
       \multiput(80,0)(0,40){3}{\line(1,1){20}}  \multiput(80,0)(10,10){3}{\line(0,1){80}}
    \put(80,0){\line(-1,1){80}}\put(80,0){\line(1,5){20}}\put(80,80){\line(-3,1){60}}
       \multiput(40,0)(40,40){2}{\line(-1,1){40}} 
        \multiput(80,40)(10,-30){2}{\line(1,5){10}}
        \multiput(40,80)(50,10){2}{\line(-3,1){30}}
      \end{picture}}
    \put(220,0){\begin{picture}(110,110)(0,0) \put(25,102){Grid 3:}
       \multiput(0,0)(20,0){5}{\line(0,1){80}}  \multiput(0,0)(0,20){5}{\line(1,0){80}}
       \multiput(0,80)(20,0){5}{\line(1,1){20}} \multiput(0,80)(5,5){5}{\line(1,0){80}}
       \multiput(80,0)(0,20){5}{\line(1,1){20}}  \multiput(80,0)(5,5){5}{\line(0,1){80}}
    \put(80,0){\line(-1,1){80}}\put(80,0){\line(1,5){20}}\put(80,80){\line(-3,1){60}}
       \multiput(40,0)(40,40){2}{\line(-1,1){40}} 
        \multiput(80,40)(10,-30){2}{\line(1,5){10}}
        \multiput(40,80)(50,10){2}{\line(-3,1){30}}

       \multiput(20,0)(60,60){2}{\line(-1,1){20}}   \multiput(60,0)(20,20){2}{\line(-1,1){60}} 
        \multiput(80,60)(15,-45){2}{\line(1,5){5}} \multiput(80,20)(5,-15){2}{\line(1,5){15}}
        \multiput(20,80)(75,15){2}{\line(-3,1){15}}\multiput(60,80)(25,5){2}{\line(-3,1){45}}
      \end{picture}}

    \end{picture} 
    \end{center} 
\caption{ The first three grids for the computation 
    in Tables  \ref{t-5}-\ref{t-6}.  } 
\label{grid4}
\end{figure}
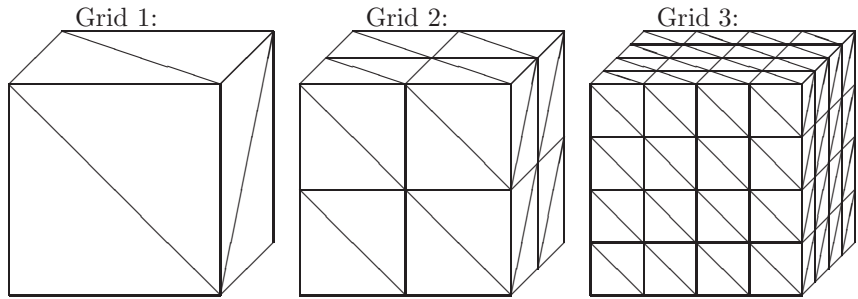

\begin{table}[H]
  \centering  \renewcommand{\arraystretch}{1.1}
  \caption{The error and the computed order of convergence for the solution \eqref{e3} on Figure \ref{grid4}, tetrahedral meshes: 
  \ (a) $P_3$ WG, $\lambda=1$;   \ (b) $P_3$ WG, $\lambda=10^4$;  \ (c) $P_4$ WG, $\lambda=1$;   \ (d) $P_4$ WG, $\lambda=10^4$.  }
  \label{t-5}
\begin{tabular}{cc|cc|cc}
\hline
Method & Grid &  $\|\b u-\b u_0\|_0$  & $O(h^r)$ & $\3bar\b u-\b u_h\3bar$  & $O(h^r)$   \\
\hline
\multirow{3}{1mm}{(a)} & 
 3 &   0.107E-01 &1.90 &   0.625E+00 &0.90 \\
& 4 &   0.267E-02 &2.00 &   0.322E+00 &0.96 \\
& 5 &   0.663E-03 &2.01 &   0.163E+00 &0.98 \\
\hline
\multirow{3}{1mm}{(b)} & 
 3 &   0.157E-01 &1.89 &   0.672E+00 &0.97 \\
& 4 &   0.389E-02 &2.01 &   0.338E+00 &0.99 \\
& 5 &   0.966E-03 &2.01 &   0.170E+00 &0.99 \\
\hline
\multirow{3}{1mm}{(c)} & 
 3 &   0.516E-03 &2.87 &   0.373E-01 &1.90 \\
& 4 &   0.662E-04 &2.96 &   0.956E-02 &1.96 \\
& 5 &   0.835E-05 &2.99 &   0.241E-02 &1.98 \\
\hline
\multirow{3}{1mm}{(d)} & 
 2 &   0.493E-02 &2.29 &   0.151E+00 &1.77 \\
& 3 &   0.604E-03 &3.03 &   0.391E-01 &1.95 \\
& 4 &   0.741E-04 &3.03 &   0.992E-02 &1.98 \\
\hline
    \end{tabular}%
\end{table}%

We next  compute the finite element solutions for \eqref{e3} on the
    nonconvex polyhedral grids shown in Figure \ref{grid5} by 
  the  $P_k$ WG finite elements, defined in \eqref{Vk}--\eqref{Vh0}, 
      for $k=1,2$ and $3$.
The results are listed in Table \ref{t-7}. 
The optimal order of convergence is achieved for all solutions in all norms. 
And the numerical results indicate the method is pressure robust.

\begin{table}[H]
  \centering  \renewcommand{\arraystretch}{1.1}
  \caption{The error and the computed order of convergence for the solution \eqref{e3} on Figure \ref{grid4}, tetrahedral meshes: 
  \ (a) $P_3$ WG, $\lambda=1$;   \ (b) $P_3$ WG, $\lambda=10^4$;  \ (c) $P_4$ WG, $\lambda=1$;   \ (d) $P_4$ WG, $\lambda=10^4$.  }
  \label{t-6}
\begin{tabular}{cc|cc|cc}
\hline
Method & Grid &  $\|\b u-\b u_0\|_0$  & $O(h^r)$ & $\3bar\b u-\b u_h\3bar$  & $O(h^r)$   \\
\hline
\multirow{3}{1mm}{(a)} & 
 2 &   0.298E-03 &3.61 &   0.136E-01 &2.74 \\
& 3 &   0.203E-04 &3.88 &   0.179E-02 &2.92 \\
& 4 &   0.131E-05 &3.96 &   0.229E-03 &2.97 \\
\hline
\multirow{3}{1mm}{(b)} & 
 2 &   0.328E-03 &3.91 &   0.140E-01 &2.80 \\
& 3 &   0.215E-04 &3.93 &   0.181E-02 &2.95 \\
& 4 &   0.138E-05 &3.96 &   0.230E-03 &2.98 \\
\hline
\multirow{3}{1mm}{(c)} & 
 2 &   0.187E-04 &4.70 &   0.106E-02 &3.80 \\
& 3 &   0.629E-06 &4.89 &   0.692E-04 &3.94 \\
& 4 &   0.202E-07 &4.96 &   0.439E-05 &3.98 \\
\hline
\multirow{3}{1mm}{(d)} & 
 1 &   0.620E-03 &0.00 &   0.156E-01 &0.00 \\
& 2 &   0.200E-04 &4.95 &   0.108E-02 &3.85 \\
& 3 &   0.650E-06 &4.95 &   0.697E-04 &3.95 \\
\hline
    \end{tabular}%
\end{table}%

\begin{figure}[H] 
\begin{center}
 \setlength\unitlength{1pt}
    \begin{picture}(220,118)(0,3)
    \put(0,0){\begin{picture}(110,110)(0,0) \put(25,102){Grid 1:}
 \multiput(0,0)(80,0){2}{\line(0,1){80}}  \multiput(0,0)(0,80){2}{\line(1,0){80}}
 \multiput(0,80)(80,0){2}{\line(1,1){20}} \multiput(0,80)(20,20){2}{\line(1,0){80}}
 \multiput(80,0)(0,80){2}{\line(1,1){20}}  \multiput(80,0)(20,20){2}{\line(0,1){80}}
  
 \multiput(0,0)(2,2){10}{\circle*{1}} \multiput(20,20)(2,0){40}{\circle*{1}}
 \multiput(20,20)(0,2){40}{\circle*{1}} 
 \put(0,0){\line(2,5){26.66}} 
 \put(26.66,66.65){\line(1,-2){26.66}} \multiput(46.66,86.65)(1,-2){26}{\circle*{1}} 
 \put(80,80){\line(-2,-5){26.66}} \multiput(20,20)(1,2.5){26}{\circle*{1}} 
 \multiput(26.66,66.65)(2,2){10}{\circle*{1}} 
  \multiput(100,100)(-1,-2.5){26}{\circle*{1}}
  \multiput(53.33,13.35)(2,2){10}{\circle*{1}} 
      \end{picture}}
      
    \put(110,0){\begin{picture}(110,110)(0,0)\put(25,102){Grid 2:}
       \multiput(0,0)(40,0){3}{\line(0,1){80}}  \multiput(0,0)(0,40){3}{\line(1,0){80}}
       \multiput(0,80)(40,0){3}{\line(1,1){20}} \multiput(0,80)(10,10){3}{\line(1,0){80}}
       \multiput(80,0)(0,40){3}{\line(1,1){20}}  \multiput(80,0)(10,10){3}{\line(0,1){80}}
       
       \multiput(-3,0)(40,0){2}{\multiput(0,0)(0,40){2}{ 
       \begin{picture}(40,40)(0,0)\put(0,0){\line(2,5){13.33}}
        \put(40,40){\line(-2,-5){13.33}} \put(13.33,33.33){\line(1,-2){13.33}} 
       \end{picture}}} 
      \end{picture}}

    \end{picture} 
    \end{center} 
\caption{ The first three grids for the computation 
    in Table \ref{t-7}.  } 
\label{grid5}
\end{figure}
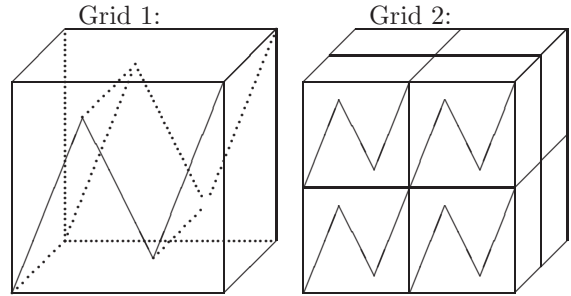

\begin{table}[H]
  \centering  \renewcommand{\arraystretch}{1.1}
  \caption{The error and the computed order of convergence for the solution \eqref{e3} on Figure \ref{grid5}, nonconvex polyhedral meshes: 
  \ (a) $P_1$ WG, $\lambda=1$;   \ (b) $P_1$ WG, $\lambda=10^4$;  \ (c) $P_2$ WG, $\lambda=1$;   \ (d) $P_2$ WG, $\lambda=10^4$; 
  \ (e) $P_3$ WG, $\lambda=1$;   \ (f) $P_3$ WG, $\lambda=10^4$.  }
  \label{t-7}
\begin{tabular}{cc|cc|cc}
\hline
Method & Grid &  $\|\b u-\b u_0\|_0$  & $O(h^r)$ & $\3bar\b u-\b u_h\3bar$  & $O(h^r)$   \\
\hline
\multirow{2}{1mm}{(a)} &  
 4 &   0.185E-01 & 2.0 &   0.217E+01 & 1.0 \\
&5 &   0.466E-02 & 2.0 &   0.109E+01 & 1.0 \\
\hline
\multirow{2}{1mm}{(b)} & 
 4 &   0.190E-01 & 2.0 &   0.217E+01 & 1.0 \\
&5 &   0.480E-02 & 2.0 &   0.109E+01 & 1.0 \\
\hline
\multirow{2}{1mm}{(c)} & 
 4 &   0.859E-03 & 3.0 &   0.159E+00 & 2.0 \\
&5 &   0.109E-03 & 3.0 &   0.401E-01 & 2.0 \\
\hline
\multirow{2}{1mm}{(d)} & 
 3 &   0.672E-02 & 2.9 &   0.620E+00 & 1.9 \\
&4 &   0.861E-03 & 3.0 &   0.159E+00 & 2.0 \\
\hline
\multirow{2}{1mm}{(e)} & 
 2 &   0.719E-02 & 3.8 &   0.497E+00 & 2.8 \\
&3 &   0.484E-03 & 3.9 &   0.656E-01 & 2.9 \\
\hline
\multirow{2}{1mm}{(f)} & 
 2 &   0.716E-02 & 3.8 &   0.496E+00 & 2.8 \\
&3 &   0.481E-03 & 3.9 &   0.656E-01 & 2.9 \\
\hline
    \end{tabular}%
\end{table}%

We finally  compute the finite element solutions for \eqref{e3} on the
    nonconvex polyhedral grids shown in Figure \ref{grid6} by 
  the  $P_k$ WG finite elements, defined in \eqref{Vk}--\eqref{Vh0}, 
      for $k=1$ and $3$.
The results are listed in Table \ref{t-8}. 
The optimal order of convergence is achieved for all solutions in all norms. 
And the numerical results indicate the method is pressure robust.

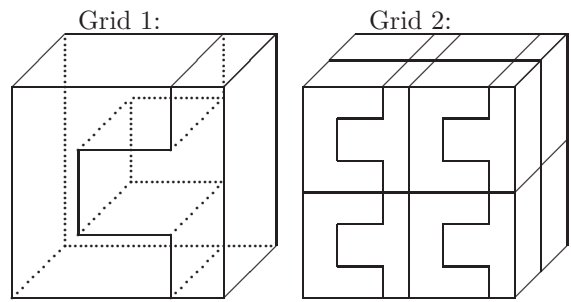
\begin{figure}[H] 
\begin{center}
 \setlength\unitlength{1pt}
    \begin{picture}(220,118)(0,3)
    \put(0,0){\begin{picture}(110,110)(0,0) \put(25,102){Grid 1:}
 \multiput(0,0)(80,0){2}{\line(0,1){80}}  \multiput(0,0)(0,80){2}{\line(1,0){80}}
 \multiput(0,80)(80,0){2}{\line(1,1){20}} \multiput(0,80)(20,20){2}{\line(1,0){80}}
 \multiput(80,0)(0,80){2}{\line(1,1){20}}  \multiput(80,0)(20,20){2}{\line(0,1){80}}
 \put(60,0){\line(0,1){24}}\put(60,80){\line(0,-1){24}}\put(60,80){\line(1,1){20}} 
 \put(25,24){\line(1,0){35}}\put(25,24){\line(0,1){32}}\put(25,56){\line(1,0){35}} 
 \multiput(0,0)(2,2){10}{\circle*{1}} \multiput(20,20)(2,0){40}{\circle*{1}}
 \multiput(20,20)(0,2){40}{\circle*{1}} \multiput(60,0)(2,2){10}{\circle*{1}} 
 \multiput(60,24)(2,2){10}{\circle*{1}} \multiput(25,24)(2,2){10}{\circle*{1}} 
 \multiput(60,56)(2,2){10}{\circle*{1}} \multiput(25,56)(2,2){10}{\circle*{1}} 
 \multiput(80,76)(0,2){12}{\circle*{1}}  \multiput(80,76)(-2,0){18}{\circle*{1}}  
 \multiput(45,44)(0,2){16}{\circle*{1}}  \multiput(45,44)(2,0){18}{\circle*{1}}  
      \end{picture}}
      
    \put(110,0){\begin{picture}(110,110)(0,0)\put(25,102){Grid 2:}
       \multiput(0,0)(40,0){3}{\line(0,1){80}}  \multiput(0,0)(0,40){3}{\line(1,0){80}}
       \multiput(0,80)(40,0){3}{\line(1,1){20}} \multiput(0,80)(10,10){3}{\line(1,0){80}}
       \multiput(80,0)(0,40){3}{\line(1,1){20}}  \multiput(80,0)(10,10){3}{\line(0,1){80}}
       \multiput(30,80)(40,0){2}{\line(1,1){20}}
       \multiput(-3,0)(40,0){2}{\multiput(0,0)(0,40){2}{ 
       \begin{picture}(40,40)(0,0)\put(30,0){\line(0,1){12}}
        \put(30,40){\line(0,-1){12}} \put(12.5,12){\line(1,0){17.5}}
         \put(12.5,28){\line(1,0){17.5}} \put(12.5,12){\line(0,1){16}}
       \end{picture}}} 
      \end{picture}}

    \end{picture} 
    \end{center} 
\caption{ The first three grids for the computation 
    in Table  \ref{t-8}.  } 
\label{grid6}
\end{figure}

\begin{table}[H]
  \centering  \renewcommand{\arraystretch}{1.1}
  \caption{The error and the computed order of convergence for the solution \eqref{e3} on Figure \ref{grid6}, nonconvex polyhedral meshes: 
  \ (a) $P_1$ WG, $\lambda=1$;   \ (b) $P_1$ WG, $\lambda=10^4$;  \ (c) $P_2$ WG, $\lambda=1$;   \ (d) $P_2$ WG, $\lambda=10^4$.  }
  \label{t-8}
\begin{tabular}{cc|cc|cc}
\hline
Method & Grid &  $\|\b u-\b u_0\|_0$  & $O(h^r)$ & $\3bar\b u-\b u_h\3bar$  & $O(h^r)$   \\
\hline
\multirow{3}{1mm}{(a)} & 
 3 &   0.520E-01 & 2.0 &   0.386E+01 & 1.0 \\
& 4 &   0.127E-01 & 2.0 &   0.191E+01 & 1.0 \\
& 5 &   0.314E-02 & 2.0 &   0.952E+00 & 1.0 \\
\hline
\multirow{3}{1mm}{(b)} & 
 2 &   0.309E+00 & 1.9 &   0.104E+02 & 1.0 \\
& 3 &   0.808E-01 & 1.9 &   0.517E+01 & 1.0 \\
& 4 &   0.207E-01 & 2.0 &   0.252E+01 & 1.0 \\
\hline
\multirow{3}{1mm}{(c)} & 
 2 &   0.289E-01 & 2.9 &   0.135E+01 & 1.9 \\
& 3 &   0.341E-02 & 3.1 &   0.336E+00 & 2.0 \\
& 4 &   0.408E-03 & 3.1 &   0.835E-01 & 2.0 \\
\hline
\multirow{3}{1mm}{(d)} & 
 1 &   0.355E+00 & 0.0 &   0.811E+01 & 0.0 \\
& 2 &   0.421E-01 & 3.1 &   0.166E+01 & 2.3 \\
& 3 &   0.487E-02 & 3.1 &   0.389E+00 & 2.1 \\
\hline
    \end{tabular}%
\end{table}%

\end{document}